\documentclass{amsart}

\usepackage{amsmath}
\usepackage{amsthm}
\usepackage{hyperref}
\usepackage{amsfonts,graphics,amsthm,amsfonts,amscd,latexsym}
\usepackage{epsfig}
\usepackage{flafter}
\usepackage{mathtools}
\usepackage{comment}
\usepackage{stmaryrd}

\usepackage{mathabx,epsfig}

\hypersetup{
    colorlinks=true,    
    linkcolor=blue,          
    citecolor=blue,      
    filecolor=blue,      
    urlcolor=blue           
}
\usepackage{tikz}
\usetikzlibrary{graphs,positioning,arrows,shapes.misc,decorations.pathmorphing}

\tikzset{
    >=stealth,
    every picture/.style={thick},
    graphs/every graph/.style={empty nodes},
}

\tikzstyle{vertex}=[
    draw,
    circle,
    fill=black,
    inner sep=1pt,
    minimum width=5pt,
]
\usepackage[position=top]{subfig}
\usepackage{amssymb}
\usepackage{color}

\calclayout

\usetikzlibrary{decorations.pathmorphing}
\tikzstyle{printersafe}=[decoration={snake,amplitude=0pt}]

\newcommand{\rank}{\operatorname{rank}}

\newcommand{\supp}{\operatorname{supp}}

\newcommand{\pp}{\mathbb{P}}

\newcommand{\qq}{\mathbb{Q}}

\newcommand{\rr}{\mathbb{R}}

\def\O#1.{\mathcal {O}_{#1}}			
\def\pr #1.{\mathbb P^{#1}}				
\def\af #1.{\mathbb A^{#1}}			
\def\ses#1.#2.#3.{0\to #1\to #2\to #3 \to 0}	
\def\xrar#1.{\xrightarrow{#1}}			
\def\K#1.{K_{#1}}						
\def\bA#1.{\mathbf{A}_{#1}}			
\def\bM#1.{\mathbf{M}_{#1}}				
\def\bL#1.{\mathbf{L}_{#1}}				
\def\bB#1.{\mathbf{B}_{#1}}				
\def\bK#1.{\mathbf{K}_{#1}}			
\def\subs#1.{_{#1}}					
\def\sups#1.{^{#1}}

\usepackage{tikz}
\usetikzlibrary{matrix,arrows,decorations.pathmorphing}

\newtheorem{introcon}{Conjecture}

  \newtheorem{introthm}{Theorem}

  \newtheorem{introcor}{Corollary}

  \newtheorem{theorem}{Theorem}[section]
  \newtheorem{lemma}[theorem]{Lemma}
  \newtheorem{proposition}[theorem]{Proposition}

  \newtheorem{definition}[theorem]{Definition}
  \newtheorem{example}[theorem]{Example}

  \newtheorem{question}[theorem]{Question}

\newtheorem{remark}[theorem]{Remark}

\theoremstyle{remark}

\numberwithin{equation}{section}

\usepackage[all]{xy}

\begin{document}

\title[Generalized Complexity of Surfaces]{Generalized Complexity of Surfaces}

\author[Y. Gongyo]{Yoshinori Gongyo}
\address{Graduate School of Mathematical Sciences, The University of Tokyo, 3-8-1 Komaba,
Meguro-ku, Tokyo, 153-8914, Japan.}
\email{gongyo@ms.u-tokyo.ac.jp}

\author[J.~Moraga]{Joaqu\'in Moraga}
\address{UCLA Mathematics Department, Box 951555, Los Angeles, CA 90095-1555, USA
}
\email{jmoraga@math.ucla.edu}

\subjclass[2020]{Primary 14E30, 
Secondary 14J32, 14B05.}

\maketitle

\begin{abstract}
In this article, we introduce the generalized complexity of a generalized Calabi--Yau pair $(X,B,\bM.)$.
This invariant compares the dimension of $X$
and Picard rank of $X$ 
with the sum of the coefficients of $B$
and $\bM.$.
It generalizes the complexity introduced by Shokurov.
We show that a generalized log Calabi--Yau pair $(X,B,\bM.)$ of dimension $2$ with
generalized complexity $0$ 
satisfies that $X$ is toric.
This generalizes a result due to Brown, McKernan, Svaldi, and Zhong in the case of surfaces. 
Furthermore, we show that a generalized klt log Calabi--Yau surface $(X,\bM.)$ with generalized complexity $0$ satisfies that $X\simeq \pp^2$ or $X\simeq \pp^1\times \pp^1$. Thus, this invariant interpolates
between the characterization of toric varieties
and the Kobayashi-Ochiai Theorem.
As an application, we show that $3$-fold singularities with generalized complexity $0$ are toric.
Furthermore, we show a local version
of Kobayashi-Ochiai Theorem in dimension $3$.
\end{abstract}

\setcounter{tocdepth}{1} 
\tableofcontents

\section{Introduction}
Fano varieties and Calabi--Yau varieties
are two of the three building blocks of algebraic varieties. 
In many cases, 
in order to study a Fano variety,
we endow it with
a log Calabi--Yau structure.
This means that we equip $X$ with the structure of a pair $(X,B)$ with {\em log canonical} singularities for which $K_X+B\equiv 0$.
For instance, it is natural to consider
the log Calabi--Yau pair
$(\pp^n,\sum_{i=1}^n H_i)$ where the $H_i$'s are the hyperplane coordinates. 
In this case, the negativity of the divisor
$K_X$ is reflected on the positivity of the boundary divisor $B$. 
The projective space $\pp^n$ can be characterized in terms of the positivity of the divisor $B$.
The following is the well-known Kobayashi-Ochiai Theorem (see~\cite{KO75}).

\begin{introthm}
\label{introthm:KO} 
Let $(X,B)$ be a $n$-dimensional log Calabi--Yau pair.
Assume that $B=\sum_{i=1}^r b_iB_i$ where each $B_i$ is an ample Cartier divisor. 
Then, we have that $\sum_{i=1}^r b_i \leq n+1$.
Furthermore, if $\sum_{i=1}^r b_i > n$, then 
$X\simeq \pp^n$ and $B_i\sim H$ for each $i$, where
$H$ is a hyperplane section.
\end{introthm}

In a similar vein, 
toric varieties can be characterized by measuring the positivity of the boundary $B$.
In this case, we need to consider the difference between the sum of the coefficient of $B$ and $\dim X+\rho(X)$.
This invariant, known as the {\em complexity},
was introduced by Shokurov in its study of complements for surfaces~\cite{Sho92}.
Toric varieties can be characterized in terms of the complexity.
The following theorem is due to Brown, McKernan, Svaldi, and Zhong~\cite{BMSZ18}.

\begin{introthm}
\label{introthm:charct-toric}
Let $(X,B)$ be a $n$-dimensional log Calabi--Yau pair.
Assume that $B=\sum_{i=1}^r b_iB_i$, where each $B_i$ is a Weil divisor. 
Then, we have that $\sum_{i=1}^r b_i \leq \dim X +\rho (X)$. 
Furthermore, if $\sum_{i=1}^r b_i > \dim X + \rho (X) -1$, then $X$ is a toric variety. 
\end{introthm}

The previous theorem is known as the characterization
of toric varieties via the complexity.
The previous theorem admits a local analog
that characterizes toric singularities~\cite{MS21}.
However, we are not aware of a local analog of the Kobayashi-Ochiai Theorem characterizing smooth points. 

Our first aim is to propose a conjecture that interpolates between the characterization
of toric varieties via complexity
and the Kobayashi-Ochiai Theorem.
First, we observe
that Theorem~\ref{introthm:KO} assumes the $B_i$'s are ample Cartier divisors
while Theorem~\ref{introthm:charct-toric} only asks the $B_i$'s to be prime divisors. 
In other words, the divisors $B_i$'s in the statement of Kobayashi-Ochiai can be regarded as a moduli divisor, i.e.,
the push-forward of a nef Cartier divisor from a higher birational model.
Moduli divisors appeared in birational geometry 
in the study of the canonical bundle formula~\cite{Amb06}.
In~\cite{BZ16}, Birkar and Zhang introduced the concept of generalized pairs, which enriches the log pair structure with an additional moduli divisor. 
The concept of generalized pairs was used by Birkar in his proof of the boundedness of Fano varieties~\cite{Bir19}.
Since then, generalized pairs became a fundamental object in birational geometry~\cite{Bir20b}.

We refer the reader to Definition~\ref{def:gen-pair}
for the concept of generalized pair
and to Definition~\ref{def:gen-lcy} 
for the concept of generalized log Calabi--Yau pair.
Throughout this article, we only consider moduli divisors which are non-negative linear combinations of nef Cartier divisors (see Remark~\ref{rem:nqc-pairs}).
Given a generalized pair $(X,B,\bM.)$, we write $|B|$ for the sum of the coefficients in the prime decomposition of $B$.
Analogously, we write $|\bM.|$ for the sum of the coefficients of $\bM.$ in its decomposition as sum of nef Cartier divisors in a model where it descends.
Henceforth, it is natural to consider the $B_i$'s in Kobayashi-Ochiai Theorem as moduli divisors
and the $B_i$'s of the characterization of toric varieties as boundary divisors. 
This observation leads to the following conjecture:

\begin{introcon}
\label{conj:gen-complexity}
Let $(X,B,\bM.)$ be a generalized log Calabi--Yau pair.
Then the following statements hold: 
\begin{enumerate}
    \item The inequality $\dim X + \rho(X)-|B|-|\bM.|\geq 0$ holds.
    \item If the equality holds in $(1)$,
    then $(X,\lfloor B\rfloor)$ is toric, and 
    \item If $B=0$, $\bM.$ descends on $X$, and the equality holds in $(1)$,  
    then $X$ is a product of projective spaces.
\end{enumerate}
\end{introcon}

Observe that Kobayashi-Ochiai Theorem can be regarded as the Picard rank one case of $(3)$ in the previous conjecture.
Our first theorem states that 
the previous conjecture holds true in dimension $2$. Indeed, we can prove a stronger result in this direction.

\begin{introthm}
\label{introthm:surface}
Let $(X,B,\bM.)$ be a generalized log Calabi--Yau pair of dimension $2$.
Then the following statements hold: 
\begin{enumerate}
    \item The inequality $\dim X + \rho(X)-|B|-|\bM.|\geq 0$ holds.
    \item If the equality holds in $(1)$,
    then $(X,\lfloor B\rfloor)$ is toric.
    \item If $B=0$ and the equality holds in $(1)$, then $X\simeq \pp^2$, $X\simeq \pp^1\times \pp^1$, or $X\simeq F_n$.\footnote{The variety $F_n$ denotes the contraction of the $(-n)$-curve in the Hirzebruch surface $\Sigma_n$.}
    \item If $B=0$, $(X,\bM.)$ is generalized klt, and equality holds in $(1)$, then 
    $X$ is a product of projective spaces.
\end{enumerate}
In particular, if $B=0$, $\bM.$ descends on $X$, and the equality holds in $(1)$, then
$X$ is a product of projective spaces.
\end{introthm}

As an application of this theorem, 
we show that $3$-fold singularities 
with generalized complexity zero are formally toric.

\begin{introthm}
\label{introthm:3-fold-sing}
Let $(X,B,\bM.;x)$ be a generalized 
log canonical singularity of dimension $3$.
Assume that $(X;x)$ is klt.
Then the following statements hold:
\begin{enumerate}
    \item The inequality $\dim X +\rho(X_x)-|B|-|\bM.|\geq 0$ holds.
    \item If the equality holds in $(1)$, then $(X,\lfloor B\rfloor)$ is formally toric at $x$. 
\end{enumerate}
\end{introthm}

In the context of Conjecture~\ref{conj:gen-complexity}, if we ask 
$\bM.$ to be a combination of ample Cartier divisors on a model where it descends, 
then we have better control on $|\bM.|$ as shown in the next theorem. 
The following theorem can be regarded as a version
of Kobayashi-Ochiai theorem
for generalized pairs. 

\begin{introthm}\label{introthm:KO-global}
Let $(X,B,\bM.)$ be a generalized log Calabi--Yau pair of dimension $n$. 
Assume the following conditions hold:
\begin{itemize}
\item the b-nef divisor $\bM.$ descends on a smooth variety $X'$, and 
\item we can write $\bM X'.=\sum_{i=1}^k \lambda_i M_i$ where each $M_i$ is an ample Cartier divisor
and each $\lambda_i$ is non-negative.
\end{itemize}
Then, we have that:
\begin{enumerate}
\item The inequality 
$\sum_{i=1}^k \lambda_i \leq n+1$ holds. 
\item If $\sum_{i=1}^k \lambda_i=n+1$, then $B=0$, 
$X'\simeq \pp^n$,
and $X'\rightarrow X$ is the identity.
\end{enumerate} 
\end{introthm}

The previous theorem admits a local version that characterizes smooth points.
The following theorem can be regarded as a local version of Kobayashi-Ochiai Theorem.

\begin{introthm}\label{introthm:KO-local}
Let $(X,B,\bM.;x)$ be a 
$\qq$-factorial generalized log canonical singularity of dimension $3$.
Assume the following conditions hold:
\begin{itemize}
\item the b-nef divisor $\bM.$ descends on a smooth variety $X'$, and 
\item we can write
$\bM X'.=\sum_{i=1}^k \lambda_i M_i$ where each $M_i$ is an ample Cartier divisor and each $\lambda_i$ is non-negative.
\end{itemize} 
Then, we have that:
\begin{enumerate}
\item The inequality 
$\sum_{i=1}^k \lambda_i \leq 3$ holds. 
\item If $\sum_{i=1}^k \lambda_i=3$, then $B=0$, $(X;x)$ is isomorphic to a cone over $\pp^2$, and $X'\rightarrow X$ is the blow-up of the maximal ideal.
\end{enumerate} 
Furthermore, if $X$ is factorial at $x$, then the germ $(X;x)$ is smooth.
\end{introthm}

The statement of Theorem~\ref{introthm:KO-local}
does not hold if we replace 
the ampleness condition
on $M_i$ with
big and nef.
We show this in Example~\ref{ex:big-nef}.
On the other hand, 
the following corollary follows from
Theorem~\ref{introthm:surface}.

\begin{introcor}\label{introcor:descend-p2}
Let $(X,B,\bM.)$ be a generalized log Calabi--Yau pair of dimension $2$.
Let $X'$ be a model where $\bM.$ descends.
Assume that $\bM X'.=\sum_{i=1}^k \lambda_i M_i$ where each $M_i$ is a big and nef Cartier divisor
and each $\lambda_i$ is nonnegative. 
Then, we have that $\sum_{i=1}^k \lambda_i \leq 3$.
Furthermore, if $\sum_{i=1}^k \lambda_i=3$, then 
$X\simeq \pp^2$
and $\bM.$ descends on $X$.
\end{introcor}

Item $(3)$ in 
Conjecture~\ref{conj:gen-complexity} suggests that the previous corollary generalizes to every dimension.

\section{Preliminaries}
\label{sec:prel}

In this section, 
we collect some preliminary results
about generalized pairs,
generalized singularities,
the generalized complexity, 
and its behaviour under adjunction.

\subsection{Generalized pair}

In this subsection, we recall some definitions and propositions
about generalized pairs.

\begin{definition}
{\em 
Let $X$ be a normal projective variety.
A {\em b-divisor} on $X$ consists of the data of a divisor $\bM Y.$ 
on each variety $Y$ which is birational to $X$, subject to the following condition.
For every projective birational morphism
$\pi\colon Y\rightarrow Z$, we have that
$\pi_*\bM Y.=\bM Z.$
We will write $\bM.$ for the b-divisor.
The divisor $\bM X.$ is called the {\em trace} of $\bM.$ in $X$.
We may say that $\bM.$ is a b-divisor on $X$ or in the birational class of $X$.
The trivial b-divisor is the b-divisor for which is trace is the trivial Weil divisor on each model.

We say that a b-divisor $\bM.$ is {\em b-Cartier} if there exists a birational model $Y$ of $X$ 
such that $\bM Z.=\pi^*\bM Y.$ for every projective birational morphism 
$\pi\colon Z\rightarrow Y$.
In this case, we say that $\bM.$ descends on $Y$.
Note that the model where $\bM.$ descends is not unique.
Indeed, if $\bM.$ descends on $Y$, then 
it also descends on any model which admits a projective birational morphism to $Y$.
If $\bM.$ is {\em nef} on a model where it descends, then we say that $\bM.$ is a {\em b-nef divisor}.
For instance, the trivial b-divisor is a b-nef divisor.

Let $\bM.$ be a b-nef divisor on $X$ and $Y$ be a model where $\bM.$ descends.
Let $\pi\colon Y\rightarrow X$ be the associated projective birational morphism.
Let $U\subset X$ be an open set
and $U_Y$ its pre-image on $Y$.
We say that $\bM.$ descends on $U$ if 
${\pi|_U}^*(\bM X.|_U)=\bM Y.|_{U_Y}$ holds.
}
\end{definition} 

\begin{definition}
\label{def:gen-pair}
{\em 
A {\em generalized pair}
consists of a triple
$(X,B,\bM.)$ where:
\begin{itemize}
\item $X$ is a normal projective variety,
\item $B$ is an effective divisor on $X$, 
\item $\bM.$ is a b-nef divisor on $X$, and 
\item the divisor $K_X+B+\bM X.$ is $\rr$-Cartier. 
\end{itemize}
If $\bM.$ is the trivial b-nef divisor, then we just write $(X,B)$ and call this a {\em log pair} or simply a {\em pair}.
We call $B$ the {\em boundary divisor}
and $\bM.$ the {\em moduli divisor}.
}
\end{definition}

\begin{definition}
\label{def:gen-lc}
{\em 
Let $(X,B,\bM.)$ be a generalized pair.
Let $Y\rightarrow X$ be a projective birational morphism.
We can write
\[
K_Y+B_Y+\bM Y.=\pi^*(K_X+B+\bM X.),
\]
for some divisor $B_Y$ on $Y$.
The {\em generalized log discrepancy}
of $(X,B,\bM.)$ at a prime divisor $E\subset Y$, denoted by $a_E(X,B,\bM.)$
is defined to be $1-{\rm coeff}_E(B_Y)$.

A generalized pair $(X,B,\bM.)$ is said to be {\em generalized log canonical} (or glc for short) if all its generalized log discrepancies are non-negative.
A generalized pair $(X,B,\bM.)$ is said to be {\em generalized Kawamata log terminal} (or gklt for short) if all its generalized log discrepancies are positive.
A generalized pair $(X,B,\bM.)$ is said to be {\em generalized terminal} if all its generalized log discrepancies are strictly larger than $1$.
In the previous definitions, if the analogous statement holds for $\bM.=0$, then we drop the word ``generalized".

Let $(X,B,\bM.)$ be a generalized log canonical pair.
A {\em generalized log canonical place} of $(X,B,\bM.)$ is a divisor $E$ over $X$ for which $a_E(X,B,\bM.)=1$.
A {\em generalized log canonical center} of
$(X,B,\bM.)$ is the image on $X$ of a log canonical place of $(X,B,\bM.)$.

We say that $(X,B,\bM.)$ is {\em generalized divisorially log terminal} (or gdlt for short) if the following conditions are satisfied:
\begin{enumerate}
\item the pair $(X,B,\bM.)$ is generalized log canonical,
\item there exists an open subset $U\subset X$ for which
$(U,\lfloor B\rfloor)$ is simple normal crossing, 
\item every generalized log canonical center of $(X,B,\bM.)$ intersects $U$ non-trivially, 
\item $\bM.$ descends on $X$ on the neighborhood of every strata of $\lfloor B\rfloor$.
\end{enumerate} 
}
\end{definition}

The following lemma states that every generalized log canonical pair can be turned into a gdlt pair
by extracting certain log canonical places.

\begin{lemma}\label{lem:dlt-mod}
Let $(X,B,\bM.)$ be a generalized log canonical pair.
There exists a projective birational morphism
$Y\rightarrow X$ satisfying the following conditions:
\begin{itemize}
\item The variety $Y$ is $\qq$-factorial,
\item the generalized pair
$(Y,B_Y,\bM.)$ defined by
$K_Y+B_Y+\bM Y.=\pi^*(K_X+B+\bM X.)$ is generalized dlt, and 
\item every prime divisor $E$ on $Y$ exceptional over $X$ satisfies that 
$a_E(X,B,\bM.)=1$.
\end{itemize}
\end{lemma} 

\begin{definition}
{\em 
The generalized pair $(Y,B_Y,\bM.)$ produced by the previous lemma is called a {\em generalized dlt modification} of the pair $(X,B,\bM.)$.
}
\end{definition} 

\begin{definition}
\label{def:gen-lcy}
{\em
Let $(X,B,\bM.)$ be a generalized log canonical pair.
We say that $(X,B,\bM.)$ is {\em generalized log Calabi--Yau} if
$K_X+B+\bM X.\sim_\qq 0$.
We say that $(X,B,\bM.)$ is {\em generalized Fano type} if 
there exists a big boundary $\Gamma$ such that $(X,B+\Gamma,\bM.)$ is
generalized log Calabi--Yau and generalized klt.
In the case that $\bM.=0$, then we drop
the word generalized from the previous definitions.
}
\end{definition}

We finish this subsection by introducing two lemmas that will be used in the proof of the main theorem.
We refer the reader to~\cite[Lemma 2.10]{FFMP22}.

\begin{lemma}\label{lem:FT-blow-up}
Let $(X,B,\bM.)$ be a generalized Fano type pair.
Let $\Gamma$ be an effective divisor for which
$(X,B+\Gamma,\bM.)$
is generalized log Calabi--Yau.
Let $Y\rightarrow X$ be a projective birational morphism that only extracts
divisors which generalized log discrepancy
in the interval $[0,1)$ with respect to
$(X,B+\Gamma,\bM.)$.
Then, $Y$ is of Fano type.
\end{lemma}

\begin{lemma}\label{lem:push-torsion}
Let $\pi\colon X\rightarrow Y$ be a projective birational morphism.
Let $\bM.$ be a b-nef divisor on $X$.
Assume that $\bM X.$ is not torsion,
then the trace of $\bM.$ on $Y$ is not torsion.
\end{lemma}

\begin{proof}
Let $\phi \colon X'\rightarrow X$ be a model where $\bM.$ descends.
By the negativity lemma, we can write
$\phi^*\bM X.=\bM X'.+F$ where $F$ is an effective divisor. 
Since $\bM X.$ is not torsion, we may find a movable curve $C$ on $X$ for which $\bM X.\cdot C\neq 0$. 
We may choose $C$ so that it avoids
$\phi(F)$. 
Thus, if $C'$ is the strict transform of $C$ on $X'$, then we have that 
$\bM X'.\cdot C'\neq 0$.
We conclude that $\bM X'.$ is not torsion. 
In particular, $\bM X'.$ is a nef Cartier divisor which is not torsion.

Assume that $\bM Y.$ is torsion.
We will lead to a contradiction.
Let $\pi'\colon X'\rightarrow Y$ be the associated projective birational morphism.
By the negativity lemma, we can write 
${\pi'}^*\bM Y.=\bM X'. + E$ where $E$ is an effective exceptional divisor.
If $E=0$, then $\bM X'.$ is torsion, leading to a contradiction.
If $E$ is a non-trivial effective divisor, we can find a curve $C'$ in $X'$ such that $E\cdot C'>0$.
The intersection of ${\pi'}^*\bM Y.$ with $C'$ must be trivial, so the intersection of $\bM X'.$ with $C'$ must be negative. This is a contradiction, as $\bM X'.$ is nef. This finishes the proof.
\end{proof} 

\subsection{Generalized complexity}
In this subsection, we introduce the generalized complexity of a generalized pair.

\begin{definition}
{\em
Let $X$ be a normal projective variety.
An {\em orbifold structure} of $X$ is a function 
$n$ from the set $X^1$ of prime divisors on $X$ to the natural numbers
which is one for all but finitely many prime divisors. 
The value of $n$ at a prime divisor $P$ is denoted by $n_P$.
An {\em orbifold Weil divisor} on $X$ is a divisor on $X$ that is a finite formal combination of the $\qq$-divisors $P/n_P$ for each $P$ prime on $X$.
}
\end{definition}

\begin{definition}
{\em 
Let $X$ be a normal projective variety and $B$ an effective divisor on $X$.
A {\em decomposition} of $B$ is a finite sum of the form
\[
\Sigma_B=\sum_{P\in X^1} \left(1-\frac{1}{n_P}\right)P
+ \sum_{i=1}^k a_i B_i \leq B. 
\]
where each $B_i$ is an orbifold Weil divisor on $X$.
The {\em norm of the decomposition} $\Sigma_B$, denoted by $|\Sigma_B|$
is the nonnegative number $\sum_{i=1}^k a_i$.
The divisors $B_i$ are called the {\em components} of the decomposition $\Sigma_B$.
}
\end{definition}

\begin{definition}
{\em 
Let $X$ be a normal projective variety.
Let $\bM.$ be a b-nef divisor on $X$.
Let $\pi\colon Y\rightarrow X$ be a model where $\bM.$ descends.
A {\em decomposition} of $\bM.$ is a finite sum of the form
\[
\Sigma_{\bM.} = \sum_{i=1}^\ell \lambda_i M_i \leq \bM Y.,
\] 
where each $\lambda_i$ is a non-negative real number, 
each $M_i$ is a nef divisor on $Y$, and 
no $M_i$ is torsion.
We may identify the decomposition of the b-nef divisor
with its push-forward on $X$, i.e., we may write
\[
\Sigma_{\bM X.} := \sum_{i=1}^\ell \lambda_i \pi_*M_i \leq \bM X..
\]
The {\em components} of $\Sigma_{\bM X.}$ are the divisors
$\pi_*M_1,\dots,\pi_*M_\ell$.
The {\em norm} of $\Sigma_{\bM X.}$, denoted by $|\Sigma_{\bM X.}|$ is the sum of the coefficients $\lambda_i$ for which $\pi_*M_i$ is not torsion.
}
\end{definition}

\begin{remark}
{\em
Observe that our definition of decomposition for b-nef divisors assumes that no component is torsion.
Due to Lemma~\ref{lem:push-torsion}, the norm $|\Sigma_{\bM X.}|$ is independent of the model of $X$, i.e., 
$|\Sigma_{\bM X.}|=|\Sigma_{\bM Y.}|$ for every birational model $Y$ of $X$.
Hence, we can just write $|\Sigma_{\bM.}|$ for this value.
}
\end{remark}

\begin{definition}
{\em 
Let $(X,B,\bM.)$ be a generalized pair.
A {\em decomposition} $\Sigma$ on the generalized pair consists of a pair
$(\Sigma_B,\Sigma_{\bM.})$ where $\Sigma_B$ is a decomposition of $B$
and $\Sigma_{\bM.}$ is a decomposition of $\bM.$.
The {\em norm} of $\Sigma$, denoted by $|\Sigma|$
is just the sum of the norm of $\Sigma_B$
and the norm of $\Sigma_{\bM.}$.
The {\em span} of the decomposition, denoted by $\langle \Sigma \rangle$ is defined to be
\[
\langle B_1,\dots,B_k, \pi_*M_1,\dots, \pi_*M_\ell \rangle \leqslant N^1(X)_\rr, 
\] 
where the $B_i$'s are the components of $\Sigma_B$
and the $\pi_*M_i$'s are the components of $\Sigma_{\bM.}$.
The {\em Picard rank} of the decomposition $\Sigma$, denoted by $\rho(\Sigma)$, is just the rank of the span $\langle \Sigma\rangle$ as an $\rr$-vector space.
}
\end{definition}

\begin{definition}
\label{def:decomp}
{\em 
Let $(X,B,\bM.)$ be a generalized pair with an orbifold structure $n$ 
and $\Sigma$ be a decomposition. 
The {\em orbifold complexity} of $(X,B,\bM.;\Sigma)$,
denoted by $\hat{c}(X,B,\bM.;\Sigma)$, and defined to be 
\[
\dim X +\rho(\Sigma) - |\Sigma|.
\] 
The {\em fine complexity} of a generalized pair $(X,B,\bM.)$, denoted by $\hat{c}(X,B,\bM.)$, is the minimum
among all the complexities of all
orbifold structures $n$ and decompositions of $\Sigma$.
The {\em absolute complexity} of a generalized pair $(X,B,\bM.)$,
denoted by $\bar{c}(X,B,\bM.)$, is the minimum among all the complexities with the trivial orbifold structure. 
The {\em classic complexity} of
a generalized pair $(X,B,\bM.)$,
denoted by 
$c(X,B,\bM.)$ 
is the minimum of 
\[
\dim X + \rho(X) - |\Sigma|,
\]
among all the decompositions $\Sigma$
with the trivial orbifold structure.
}
\end{definition} 

\begin{remark}
{\em 
Let $(X,B,\bM.)$ be a generalized pair.
The following inequalities follow from the definitions:
\[
c(X,B,\bM.) \geq 
\bar{c}(X,B,\bM.) \geq
\hat{c}(X,B,\bM.). 
\]
}
\end{remark}

\begin{remark}
\label{rem:nqc-pairs}
{\em
A b-nef divisor $\bM.$ admitting a non-trivial decomposition
$\Sigma_{\bM.}$ is known as ${\rm NQC}$-divisor in the literature.
Throughout this work, we only consider b-nef divisors admitting
non-trivial decompositions.
}
\end{remark}

\begin{remark}
{\em 
The classic complexity was defined by Shokurov in~\cite{Sho00}.
Shokurov expected that this invariant could characterize toric morphism. This statement was later proved in the projective case in~\cite{BMSZ18}.
One issue with the definition of the classic complexity is that it does not behave well when restricted to the general fiber of a morphism.
Meaning that, even if the classic complexity of a pair $(X,B)$ is bounded above, 
we may find a fibration $X\rightarrow Z$ so that the general fiber $(F,B_F)$ has large classic complexity. 
This issue was fixed with the definition of fine complexity in~\cite{BMSZ18}.
On the other hand, the fine complexity does not behave well under adjunction.
In~\cite{MS21}, the authors introduced the orbifold complexity to remedy this issue. 
}
\end{remark}

\subsection{Generalized complexity under adjunction} In this subsection, we study how the generalized complexity behaves under adjunction.
 
\begin{proposition}\label{prop:gen-comp-under-adj}
Let $(X,B,\bM.)$ be a generalized pair
with a decomposition $\Sigma$.
Let $S$ be a prime component of $\lfloor B\rfloor$. Assume that $S$ is normal.
Assume that the restriction
of every component of $\Sigma$ to $S$ 
is not numerically trivial.
Assume that $S$ appears with coefficient $1$ in the boundary
decomposition of $\Sigma$.
Let $(S,B_S,\bM S.)$ be the generalized pair obtained by adjunction.
Then, there exists a decomposition
$\Sigma_S$ of $(S,B_S,\bM S.)$ for which
\[
\hat{c}(S,B_S,\bM S.;\Sigma_S)
\leq 
\hat{c}(X,B,\bM.;\Sigma).
\]
\end{proposition}

\begin{proof}
Let $n$ be the orbifold structure on $X$
associated to $\Sigma$.
We write 
\[
\Sigma_B=\sum_{P\in X^1}\left(1-\frac{1}{n_P}\right)P+\sum_{i=1}^k b_iB_i,
\] 
be the corresponding decomposition of the boundary divisor. 
Let 
\[
\Sigma_{\bM.}=\sum_{i=1}^\ell \lambda_i M_i,
\]
be the decomposition of the moduli divisor. 
Let $Q\in E^1$. Observe that one of the following holds.
\begin{itemize}
    \item[(a)] There is a unique prime divisor $P\in X^1$ with $n_P>1$ containing $Q$.
    \item[(b)] If there are two prime divisors $P_1,P_2$ in $X^1$ with
    $n_{P_1}=n_{P_2}=2$ containing $Q$.
\end{itemize}
We denote by $\mathcal{S}$ the prime divisors $Q\in E^1$ satisfying $(b)$.
Given $Q\in E^1$, we denote by $i_Q$ the Cartier index of $E$ at the generic point of $Q$.
We define an orbifold structure $m$ on $S$ as follows.
We set $m_Q=i_Q$ if $Q$ is not in the support of $n$,
$m_Q=n_Qi_Q$ if $(a)$ holds, and
$m_Q=1$ if $(b)$ holds.

Let $\pi\colon Y\rightarrow X$ be a log smooth model where
$\bM.$ descends. Let $S_Y$ be the strict transform of $S$ on $Y$. 
We denote by $\pi_S\colon S_Y\rightarrow S$ the corresponding birational morphism.
We write $M_{X,i}:=\pi_*M_i$. 
Up to reordering the divisors $M_1,\dots,M_\ell$, we may assume
that $M_1|_{S_Y},\dots,M_j|_{S_Y}$ are non-torsion
and $M_{j+1}|_{S_Y},\dots,M_\ell|_{S_Y}$ are torsion.
We define a decomposition 
\[
\Sigma_{\bM S.}=\sum_{i=1}^j \lambda_i M_i|_{S_Y} \equiv 
\bM S_Y.,
\] 
of the moduli divisor of $\bM S.$.
For $i\in \{j+1,\dots,\ell\}$, we write
\[
\pi^*M_{X,i} = M_i + E_i,
\]
where $E_i$ is an effective divisor. 
We claim that $E_i\cap S_Y$ contains a non-exceptional divisor
over $S$.
Indeed, assume the opposite holds, i.e., 
$E_i\cap S_Y$ is exceptional over $S$.
Let $C$ be a general curve in $S$
and $C'$ be its strict transform on $S_Y$.
Then, we have that $C'\cdot E_i=0$.
On the other hand, since we are assuming that
$M_i|_{S_Y}$ is torsion, we have that
$C'\cdot M_i=0$.
By the projection formula, we conclude that
$M_{X,i}\cdot C=0$. 
Thus, $M_{X,i}$ intersects trivially every movable curve.
This leads to a contradiction.

Let $i_S\colon S\rightarrow X$ be the inclusion morphism.
We consider the decomposition 
\[
\Sigma_{B_S}:=
\sum_{Q\in S^1}\left( 
1-\frac{1}{m_Q} 
\right) Q
+
\sum_{Q\in \mathcal{S}} Q
+
\sum_{i=1}^k b_i B_i|_S 
+
\sum_{i=j+1}^\ell \lambda_i {\pi_S}_*(E_i|_{S_Y}).
\] 
By the generalized adjunction formula~\cite{BZ16},
we know that $\Sigma_{B_S}\leq B_S$.
On the other hand, by construction, we know that
$\Sigma_{B_S}$ is a decomposition
with respect to the orbifold structure $m$.

We consider the decomposition
$\Sigma_S$ given by
$(\Sigma_{\bM S.},\Sigma_{B_S})$.
By construction, we have that 
\[
|\Sigma_S|=|\Sigma_{\bM S.}|+|\Sigma_{B_S}| =
\left(\sum_{i=1}^j \lambda_i \right) 
+
\left( \sum_{i=1}^k b_i + \sum_{i=j+1}^\ell \lambda_i \right) + |\mathcal{S}| \geq |\Sigma|-1.
\]
On the other hand, observe that 
\[
\rho(\Sigma_S)=\rank\left( 
\langle B_1|_S,\dots,B_k|_S, 
M_{X,1}|_S,\dots, 
M_{X,j}|_S, 
{\pi_S}_*(E_{j+1}|_{S_Y}),
\dots 
{\pi_S}_*(E_\ell|_{S_Y})
\rangle 
\right).
\]
Note that ${\pi_S}_*(E_i|_{S_Y})$
is numerically equivalent to $M_{X,i}|_S$ for
every $i\in \{j+1,\dots,\ell\}$.
Indeed, this follows from the fact that
$M_i|_{S_Y}$ is torsion for every
$i\in \{j+1,\dots,\ell\}$.
Hence, we conclude that 
\[
\rho(\Sigma_S)=\rank\left(
\langle 
B_1|_S,\dots,B_k|_S, 
M_{X,1}|_S,\dots,M_{X,\ell}|_S 
\rangle 
\right) 
\leq 
\rank\left(
\langle 
B_1,\dots,B_k|_S, 
M_{X,1},\dots,M_{X,\ell}|_S 
\rangle 
\right) 
= \rho(\Sigma).
\] 
We conclude that 
\[ 
\hat{c}(S,B_S,\bM S.;\Sigma_S) = 
\dim S + \rho(\Sigma_S) - |\Sigma_S| 
\leq  
\dim X -1 + \rho(\Sigma) - (|\Sigma|-1)
= \hat{c}(X,B,\bM.;\Sigma). 
\]
This proves the proposition.
\end{proof}

\begin{remark}
\label{rem:irreducible-comp}
{\em 
In the previous proof, if $B_i|_S$ is not irreducible
for some $i$ and $\rho(\Sigma_S)=\rho(S)$, then 
we may find a decomposition 
$\Sigma_S$ for which 
\[
\hat{c}(S,B_S,\bM S.;\Sigma_S)<\hat{c}(X,B,\bM.;\Sigma).
\]
}
\end{remark}

\section{Generalized complexity for surfaces}

In this section, we prove the following theorem 
about the generalized complexity.

\begin{theorem}\label{thm:gen-log-CY-comp}
Let $(X,B,\bM.)$ be a generalized log Calabi--Yau pair of dimension $2$. The following statements hold.
\begin{enumerate}
    \item The inequality $\hat{c}(X,B,M)\geq 0$ holds.
    \item If $\hat{c}(X,B,\bM.)=0$, then $(X,\lfloor B\rfloor)$ is a toric pair.
    \item If $\hat{c}(X,\bM.)=0$, then either $X\simeq \pp^2$, $X\simeq \pp^1\times \pp^1$, or $X\simeq F_n$.
    \item If $\hat{c}(X,\bM.)=0$ and $(X,\bM.)$ is generalized klt, then $X\simeq \pp^2$ or $X\simeq \pp^1 \times \pp^1$.
\end{enumerate}
In particular, if $\hat{c}(X,\bM.)=0$ and $\bM.$ descends on $X$, then $X$ is isomorphic
to a product of projective spaces.
Furthermore, if $\hat{c}(X,B,\bM.)=0$, then the components of $\Sigma$ span $N^1_\rr(X)$, i.e., $\rho(X)=\rho(\Sigma)$.
\end{theorem} 

In order to prove this theorem, we will proceed in four steps. 
In subsection~\ref{subsec:FT}, we will prove $(1)$ and $(2)$ when $X$ is a Fano type variety.
In subsection~\ref{subsec:spanning-rho}, we will prove $(1)$ and $(2)$ 
for generalized log Calabi--Yau pairs $(X,B,\bM.)$ with a decomposition $\Sigma$ satisfying $\rho(\Sigma)=\rho(X)$.
In subsection~\ref{subsec:reduct-FT}, we will reduce the statements $(1)$ and $(2)$ in the general setting  to either of the previous cases.
Finally, in subsection~\ref{subsec:toric case}, we will achieve $(3)$ and $(4)$ by an argument using the geometry of toric surfaces. 

\subsection{Fano type case}\label{subsec:FT}
In this subsection, we show that $(1)$ and $(2)$ in Theorem~\ref{thm:gen-log-CY-comp} are valid for Fano type surfaces $X$.
We will start by introducing some lemmas. 
First, we show a version of Kawamata's non-vanishing for surfaces in the context of generalized pairs.

\begin{lemma}\label{lem:non-vanishing-surfaces}
Let $X$ be a Fano type surface.
Let $(X,B,\bM.)$ be a generalized log Calabi--Yau pair. 
Let $L$ be an ample Cartier divisor on $X$.
Then, we have that $H^0(X,\mathcal{O}_X(L))\neq 0$.
\end{lemma}

\begin{proof}
Since $X$ is Fano type, we may find a big boundary $\Gamma$
such that $(X,\Gamma)$ is klt and $K_X+\Gamma\sim_\qq 0$. 
Now, the pair
$(X,(1-\epsilon)B+\epsilon\Gamma,(1-\epsilon)\bM.)$ is generalized klt and
generalized log Calabi--Yau.
Let $A$ be an ample divisor on $X$.
Choose $\delta>0$ small enough so that $L-\delta A$ is ample.
By~\cite[Lemma 3.7]{MS21}, there exists a boundary $G\sim_\qq (1-\epsilon)\bM.+\epsilon \Gamma+\delta A$ for which
$(X,(1-\epsilon)B+G)$ is a klt pair. 
By construction, we have that 
\[
L -(K_X+(1-\epsilon)B+G) 
\]
is an ample divisor. 
By~\cite[Theorem 3.1]{Kaw00}, we conclude that $H^0(X,\mathcal{O}_X(L))\neq 0$.
\end{proof}

The following lemma is an application of a vanishing theorem due to Fujino~\cite[Theorem 1.11]{Fuj12}.

\begin{lemma}\label{lem:lifting-surfaces}
Let $X$ be a Fano type surface.
Let $(X,B,\bM.)$ be a generalized log Calabi--Yau pair. 
Let $W$ be a union of generalized log canonical centers of $(X,B,\bM.)$.
Let $L$ be an ample Cartier divisor on $X$. Then, we have that
the natural homomorphism
\[
H^0(X,\mathcal{O}_X(L))
\rightarrow
H^0(W,\mathcal{O}_W(L|_W))
\]
is an epimorphism.
\end{lemma}

\begin{proof}
Let $(Y,B_Y,\bM.)$ be a generalized dlt modification
of $X$.
By Lemma~\ref{lem:FT-blow-up}, the variety $Y$ is of Fano type.
Let $\Gamma_Y$ be a $\qq$-complement of $(Y,B_Y)$.
This $\qq$-complement exists, for instance, by~\cite[Theorem 1.2]{FM20}.
Then, the pair $(Y,B_Y+\Gamma_Y)$ is log canonical
and log Calabi--Yau.
Furthermore, every generalized log canonical center of $(Y,B_Y,\bM.)$ is a 
log canonical center of $(Y,B_Y+\Gamma_Y)$.
Let $(X,B+\Gamma)$ be the push-forward of $(Y,B_Y+\Gamma_Y)$.
Then, any generalized log canonical center of
$(X,B,\bM.)$ is a log canonical center of $(X,B+\Gamma)$.
Furthermore, the pair $(X,B+\Gamma)$ is log canonical and log Calabi--Yau.
Then, the statement follows from~\cite[Theorem 1.11]{Fuj12}.
\end{proof}

The following lemma states that the base locus of a semiample Cartier divisor on a Fano type surface contains
no generalized non-klt centers of complements.

\begin{lemma}\label{lem:not-containing-non-klt}
Let $X$ be a Fano type surface. 
Let $(X,B,\bM.)$ be a generalized log Calabi--Yau pair.
Let $L$ be a semiample Cartier divisor on $X$ which is not numerically trivial. 
Then, we have that $H^0(X,\mathcal{O}_X(L))\neq 0$
and a general element of $H^0(X,\mathcal{O}_X(L))$
contains no generalized non-klt centers of $(X,B,\bM.)$.
\end{lemma}

\begin{proof}
The first statement is straightforward from Lemma~\ref{lem:non-vanishing-surfaces} by considering the morphism $\pi\colon X\rightarrow X_0$ induced by $L$.
We show the second statement.
First, assume that $X_0$ is a normal curve.
Then, $X_0\simeq \pp^1$.
The statement is clear as there are only finitely many points on $\pp^1$ whose fibers contain a glcc of $(X,B,\bM.)$.
Now, assume that $X_0$ is a surface. 
Then, $X_0$ is a Fano type surface, being the image of a Fano type surface.
Let $(X_0,B_0,\bM .)$ be the generalized log Calabi--Yau pair induced on $X_0$.
Note that $X\rightarrow X_0$ is a Fano type morphism, 
then $L$ is the pull-back of an ample Cartier divisor $H_0$ on $X_0$ by the relative cone theorem.
Hence, we may write $\pi^*H_0=L$.
By Lemma~\ref{lem:lifting-surfaces}, we know that a general element of $H^0(X_0,\mathcal{O}_{X_0}(H_0))$ does not contain
a generalized non-klt center of $(X_0,B_0,\bM.)$.
Indeed, if $W$ is a generalized log canonical center of $(X_0,B_0,\bM.)$, then we can find a generalized log canonical center $W_0\subseteq W$ which is either a point or a smooth curve.
so
$H^0(W_0,\mathcal{O}_{X_0}(H_0|_{W_0}))$ admits a non-trivial section.
We conclude that a general element
of $H^0(X,\mathcal{O}_X(L))$ does not contain $W_0$.
We deduce that a general element of $H^0(X,\mathcal{O}_X(L))$ does not contain generalized non-klt centers of $(X_0,B_0,\bM.)$.
\end{proof}

Now, we turn to prove the main statement of this section.

\begin{proposition}\label{prop:theorem-FT-case}
Let $X$ be a Fano type surface.
Let $(X,B,\bM.)$ be a generalized log Calabi--Yau pair.
Then, the following statements hold:
\begin{enumerate}
    \item The inequality $\hat{c}(X,B,\bM.)\geq 0$ holds.
    \item If $\hat{c}(X,B,\bM.)=0$, then $(X,\lfloor B\rfloor)$ is a toric pair.
\end{enumerate}
Furthermore, the components of $\Sigma$ span $N^1_\rr(X)$, i.e., 
$\rho(\Sigma)=\rho(X)$.
\end{proposition}

\begin{proof}
It suffices to show both statements for a fixed
orbifold structure $n$ on $X$ 
and a decomposition $\Sigma$ on the generalized pair.
Let $Y$ be a smooth model where $\bM.$ descends.
Let $\Sigma_{\bM.}:=\sum_{i=1}^k \lambda_iM_i\leq \bM Y.$
be the induced decomposition of the b-nef divisor.
We will proceed with the proof in four steps.\\ 

\noindent{\it Step 1:} In this step, we show that there exists a projective birational morphism $\phi\colon Y\rightarrow Z$ over $X$
for which $Z$ is Fano type and $\phi_*M_i$ is a nef Cartier divisor for each $i$.\\ 

Let $Z$ be a projective birational model over $X$
that extracts all the divisors with generalized log discrepancy strictly less than $1$ with respect to $(X,B,\bM.)$. 
We may assume that $Z$ admits a projective birational morphism from $Y$.
Furthermore, we may assume that $Y$ is a minimal resolution of $X$.
Let $\phi\colon Y\rightarrow Z$ be the induced projective birational morphism.
By Lemma~\ref{lem:FT-blow-up}, we know that $Z$ is a Fano type variety.
We claim that each divisor $\phi_*M_i$ is nef Cartier.
It is clear that such divisors are nef.
It suffices to prove that these divisors are Cartier.
Since $Y\rightarrow X$ is a minimal resolution,
all the divisors extracted on $Y$ have generalized log discrepancy at most $1$ with respect to $X$.
By construction of $Z$, we know that $\phi$ only extracts
divisors with generalized log discrepancy equal to $1$ with respect to $(X,B,\bM.)$.
Furthermore, these divisors must have log discrepancy equal to $1$ with respect to $X$. 
Hence, we conclude that $\phi^*\phi_*M_i=M_i$ for each $i$.
This implies that each $\phi_*M_i$ is Cartier.
We let $(Z,B_Z,\bM Z.)$ be the log pull-back of $(X,B,\bM.)$ to $Z$.
From now on, we write $M_{Z,i}=\phi_*M_i$ for each $i\in \{1,\dots,k\}$.\\

\noindent{\it Step 2:} In this step, we show that each nef Cartier divisor
$M_{Z,i}$ is linearly equivalent to an effective divisor.\\ 

The variety $Z$ is Fano type. 
Hence, it is a Mori dream space.
On the other hand, $M_{Z,i}$ is nef Cartier.
In particular, it is semiample. 
Let $\pi_i\colon Z\rightarrow Z_i$ be the projective morphism defined by $M_{Z,i}$.
Note that $Z\rightarrow Z_i$ is Fano type. 
Then, by the relative cone theorem
there is an ample Cartier divisor $H_i$ on $Z_i$ such that $M_{Z,i}=\pi_i^*H_i$.
Note that we are assuming $M_i$ is not a torsion element, 
so $M_{Z,i}$ are not torsion elements.
In particular, $Z_i$ is either a normal surface or a normal curve.
If $Z_i$ is a curve, then it is clear that $H_i$ is linearly equivalent to an effective divisor and so it is $M_{Z,i}$.
If $Z_i$ is a surface, we denote by $B_{Z_i}$ the push-forward of $B_Z$ to $Z_i$.
Then, $Z_i$ is Fano type and 
$(Z_i,B_{Z_i},\bM Z_i.)$ is a generalized log Calabi--Yau pair.
By Lemma~\ref{lem:non-vanishing-surfaces}, we conclude that
$H^0(Z_i,\mathcal{O}_{Z_i}(H_i))\neq 0$
and hence the same holds for $M_{Z,i}$.
From now on, we will let
$\Gamma_{Z,i} \in H^0(Z,\mathcal{O}_Z(M_{Z,i}))$ a general element.\\

\noindent{\it Step 3:} In this step, we show that there exists an effective divisor $\Gamma$ on $X$ with a decomposition $\Sigma_\Gamma$ satisfying the following conditions:
\begin{itemize}
    \item the pair $(X,B+\Gamma)$ is log Calabi--Yau,
    \item we have that $\langle \Sigma\rangle=\langle \Sigma_B,\Sigma_\Gamma\rangle$, and 
    \item we have that $|\Sigma|=|\Sigma_B|+|\Sigma_\Gamma|$.\\
\end{itemize}

Let $\Sigma_{\Gamma_Z}:=\sum_{i=1}^k \lambda_i \Gamma_{Z,i}=\Gamma_Z$
be an effective divisor on $Z$
with the given decomposition.
It suffices to show that 
$(Z,B_Z+\Gamma_Z)$ is a log Calabi--Yau pair with log canonical singularities.
Indeed, if this is the case, then we can take $\Sigma$ the decomposition induced on $X$ (by pushing forward), and observe that it satisfies all the required conditions.

By construction, we have that $K_Z+B_Z+\Gamma_Z\sim_\rr 0$.
By contradiction, assume that the log pair
$(Z,B_Z+\Gamma_Z)$ is not log canonical.
For each $i\in \{1,\dots,k\}$, we denote by $\bM i.$
the b-nef divisor induced by $M_i$.
Let $\ell \in \{1,\dots,k\}$ be the smallest positive integer for which the generalized pair 
\[
\left( Z, B_Z+\sum_{i=1}^{\ell} \lambda_i \Gamma_{Z,i},\sum_{i=\ell+1}^k \lambda_i \bM i. 
\right) 
\]
is not generalized log canonical.
Then, we have that the generalized pair
\[
\left( Z, B_Z +\sum_{i=1}^{\ell-1}\lambda_i \Gamma_{Z,i},\sum_{i=\ell}^k \lambda_i \bM_i.
\right)
\] 
is generalized log canonical and log Calabi--Yau.
Let $t_0<1$ be the largest positive real number for which the generalized pair 
\begin{equation}\label{eq:gen-pair} 
\left(
Z,B_Z+\sum_{i=1}^{\ell-1} \lambda_i \Gamma_{Z,i} + 
\lambda_\ell t_0\Gamma_{Z,\ell},
(\lambda_\ell - t_0\lambda_\ell)\bM \ell. + \sum_{i=\ell+1}^k \lambda_i \bM i.
\right)
\end{equation} 
is generalized log canonical.
Let $W$ be a generalized log canonical center of~\eqref{eq:gen-pair}.
By construction, the generalized log canonical center $W$ must be contained in $\Gamma_{Z,\ell}$. 
On the other hand, we chose the element $\Gamma_{Z,\ell}$ to be general in the linear system $H^0(Z,\mathcal{O}_Z(M_{Z,i}))$.
This contradicts Lemma~\ref{lem:not-containing-non-klt}.\\

\noindent{\it Step 4:} In this step, finish the proof of the statement.\\

Set $\Delta:=B+\Gamma$.
Consider the decomposition $\Sigma_\Delta:=\Sigma_B+\Sigma_\Gamma$.
By the previous step, we know that $(X,\Delta)$ is log canonical and log Calabi--Yau.
Furthermore, we know that 
$\langle \Sigma\rangle=\langle \Sigma_\Delta \rangle$
and 
$|\Sigma_\Delta|=|\Sigma|$. 
Thus, we have that 
\begin{equation}\label{eq:complexity-comparison}
\hat{c}(X,B,\bM.;\Sigma)=
\hat{c}(X,\Delta;\Sigma_\Delta).
\end{equation} 
By~\cite[Theorem 4.2]{MS21}, the right-hand side of the equality is non-negative.
This shows $(1)$.
If the right-hand side of~\eqref{eq:complexity-comparison} is zero, then~\cite[Theorem 4.2 (1) and (2)]{MS21} implies that $(X,\lfloor \Delta \rfloor)$ is toric. Since $\lfloor B\rfloor \leq \lfloor \Delta \rfloor$, we conclude that $(X,\lfloor B\rfloor)$ is toric.
This shows $(2)$.
Finally, observe that $\rho(\Sigma_\Delta)=\rho(X)$ so 
$\rho(\Sigma)=\rho(X)$.
This finishes the proof of the proposition.
\end{proof}

\subsection{Decompositions spanning the Picard group} 
\label{subsec:spanning-rho}
In this subsection, we show that the first two propositions of Theorem~\ref{thm:gen-log-CY-comp}
hold for generalized pairs
$(X,B,\bM.)$ with decompositions $\Sigma$ for which $\rho(\Sigma)=\rho(X)$.
In this case, we say that the decomposition spans the Picard group.

\begin{proposition}\label{prop:LCY-case-with-full-span} 
Let $(X,B,\bM.)$ be a generalized log Calabi--Yau pair of dimension $2$. Let $n$ be an orbifold structure on $X$ and $\Sigma$ be a decomposition of the generalized pair $(X,B,\bM.)$. 
Assume that $\rho(\Sigma)=\rho(X)$.
Then, the following conditions hold:
\begin{enumerate}
\item The inequality $\hat{c}(X,B,\bM.;\Sigma)\geq 0$, and
\item if $\hat{c}(X,B,\bM.;\Sigma)=0$, then
$(X,\lfloor B\rfloor)$ is toric.
\end{enumerate} 
Furthermore, if the equality holds, then $\rho(\Sigma)=\rho(X)$.
\end{proposition}

In order to prove the previous proposition, we will use the following lemma which follows from the definitions (see also~\cite{MS21}).

\begin{lemma}\label{lem:decomp-contraction}
Let $(X,B,\bM.)$ be a generalized log Calabi--Yau pair of dimension $2$. Let $n$ be an orbifold structure on $X$ and $\Sigma$ be a decomposition of the generalized pair $(X,B,\bM.)$. 
Let $X\rightarrow Z$ be a divisorial contraction
and $(Z,B_Z,\bM.)$ the induced generalized pair on $Z$.
Let $C$ be the curve contracted by $X\rightarrow Z$.
Then, we can find a decomposition $\Sigma_Z$ of $(Z,B_Z,\bM.)$ such that
\begin{equation}\label{eq:lemma-contraction} 
\hat{c}(X,B,\bM.;\Sigma) \geq \hat{c}(Z,B_Z,\bM.;
\Sigma_Z).
\end{equation} 
Furthermore, assume that $\rho(\Sigma)=\rho(X)$, then the equality holds if and only if $C$ appears in $\Sigma$ with coefficient one.
\end{lemma} 

\begin{proof}
Let $C$ be the curve contracted by $X\rightarrow Z$. 
Let $\Sigma_Z$ be the decomposition induced by push-forward.
First, assume that $C$ does not appear in the boundary of the decomposition $\Sigma$.
Then, we have that
$|\Sigma_Z|=|\Sigma|$,
while $\rho(\Sigma)-1\leq \rho(\Sigma_Z)\leq \rho(\Sigma)$. Then, the inequality~\eqref{eq:lemma-contraction} holds in this case.
Assume that $C$ appears in $\Sigma$ with coefficient $a$, then $|\Sigma|=|\Sigma_Z|-b$
and $\rho(\Sigma_Z)=\rho(\Sigma)-1$. 
Thus, the inequality~\eqref{eq:lemma-contraction} holds as well.

Now, assume that $\rho(\Sigma)=\rho(X)$.
This implies that 
$\rho(\Sigma_Z)=\rho(Z)=\rho(X)-1$.
Hence, the complexity drops exactly by $1-b$ where $b$ is the coefficient of $C$ in $\Sigma$.
\end{proof} 

Now, we turn to prove the main proposition of this subsection.

\begin{proof}[Proof of Proposition~\ref{prop:LCY-case-with-full-span}]
We will show the statement by reducing it to the Fano type surface case and applying the result in the previous subsection.\\ 

\textit{Step 1:} In this step, we reduce to the case in which $(X,B,\bM.)$ is a generalized dlt surface.\\

Let $n$ be an orbifold structure on $X$.
Let $\Sigma$ be a decomposition of the generalized pair.
We write
\[
\Sigma_B =\sum_{P\in X^1}
\left(1-\frac{1}{n_P}\right)P 
+\sum_{i=1}^k b_i B_i,
\]
for the associated decomposition of the boundary divisor $B$.
We write 
$\Sigma_{\bM .}=\sum_{i=1}^\ell \lambda_i M_i$ be a decomposition of the b-nef divisor on a model where it descends.
Let $(Y,B_Y,\bM.)$ be a $\qq$-factorial generalized dlt modification of $(X,B,\bM.)$ (see Lemma~\ref{lem:dlt-mod}).
We write $\pi\colon Y\rightarrow X$ for the associated projective birational morphism.
We consider the orbifold structure $m$ on $Y$ for which $m_P=1$ if $P$ is exceptional over $X$ 
and $m_P=n_{\pi(P)}$ otherwise.
We construct a decomposition 
$\Sigma_Y$ on the generalized 
pair $(Y,B_Y,\bM.)$ as follows. 
We write 
\[
\Sigma_{B_Y} = \sum_{E\in {\rm Ex}(Y/X)} E + 
\sum_{P\in Y^1}\left( 
1-\frac{1}{m_P}
\right)P
+
\sum_{i=1}^k b_iB_{Y,i}
\]
Thus, $\Sigma_{B_Y}$ and $\Sigma_{\bM.}$ induce a decomposition
$\Sigma_Y$ on the generalized pair $(Y,B_Y;\bM.)$.
Let $\rho$ be the relative Picard rank of the morphism $Y\rightarrow X$.
Observe that 
\[
|\Sigma_Y| = |\Sigma|+\rho.
\]
On the other hand, we have that 
\[
\rho(Y)=
\rho(\Sigma_Y) 
\leq \rho(\Sigma)+\rho 
= \rho(X)+\rho.
\]
We conclude that 
\[
\hat{c}(Y,B_Y,\bM.;\Sigma_Y)
\leq 
\hat{c}(X,B,\bM.;\Sigma). 
\]
Recall that the image of a toric variety under a projective birational morphism is again toric.
Hence, replacing $(X,B,\bM.)$ with
$(Y,B_Y,\bM.)$, we may assume that the generalized pair is dlt.
In particular, we may assume that $X$ is $\qq$-factorial.\\

\textit{Step 2:} In this step, we run a minimal model program and show that certain conditions on the complexity are preserved.\\ 

We run a minimal model program for $K_X$.
If $B=0$ and $\bM.=0$, then 
the statement is trivial.
So we may assume that either the boundary divisor or the moduli divisor is non-trivial.
Hence, this minimal model program $X\rightarrow Z$ must terminate with a Mori fiber space $Z\rightarrow C$, where $C$ is either a curve or a point. 
Let $B_Z$ be the push-forward of $B$ on $Z$. 
Then, the generalized log canonical pair
$(Z,B_Z,\bM.)$ is generalized log Calabi--Yau.
By Lemma~\ref{lem:curve-contraction},
there exists a decomposition
$\Sigma_Z$ on $(Z,B_Z,\bM.)$ for which
\begin{equation}\label{eq:comp-comparison}
\hat{c}(Z,B_Z,\bM.;\Sigma_Z) 
\leq \hat{c}(X,B,\bM.;\Sigma).
\end{equation} 
Furthermore, the previous equality holds if and only if all the contracted divisors on 
$X\rightarrow Z$ are generalized log canonical centers of $(X,B,\bM.)$. 
Observe that the property $\rho(\Sigma_Z)=\rho(Z)$ holds.\\ 

\textit{Step 3:} In this step, we show that the outcome of the minimal model program is a Fano type surface and conclude the proof.\\

If $Z$ has Picard rank one, then the statement is obvious.
Hence, we may assume that we have a Mori fiber space $Z\rightarrow C$. 
Assume all the components of $B_Z$ and $\bM Z.$ are horizontal over $C$. Restricting to the general fiber we get a generalized pair
$(\pp^1,B_{\pp^1},\bM \pp^1.)$ of negative complexity, leading to a contradiction.
Hence, we may assume that there is at least one component of $B_Z$ or $\bM Z.$ that is vertical over $C$.
In particular, we have that $C\simeq \pp^1$.
Without loss of generality, we may assume that this is a component of $B_Z$ or $\bM Z.$ that we denote by $B_1$.
On the other hand, there exists a component of $B_Z$ or $\bM Z.$ which is horizontal over $C$.
Without loss of generality, we may assume that this is a component of $B_Z$ that we denote by $B_2$.

We run a $(K_Z+B_z-\epsilon B_1+\bM Z.)$-MMP
that consists of a single step $Z\rightarrow Z_0$.
First, assume that the contraction $Z\rightarrow Z_0$ is birational.
Let $(Z_0,B_{Z_0},\bM .)$ be the induced generalized log Calabi--Yau. 
Note that $Z_0$ is Fano type
as it has Picard rank one and carries
a non-trivial log Calabi--Yau pair structure.
By Lemma~\ref{lem:decomp-contraction}, 
there is a decomposition
$\Sigma_{Z_0}$ for which 
\[
\hat{c}(Z_0,B_{Z_0},\bM.;\Sigma_{Z_0}) \leq 
\hat{c}(Z,B_Z,\bM.;\Sigma_Z) \leq 
\hat{c}(X,B,\bM.;\Sigma). 
\]
By Proposition~\ref{prop:theorem-FT-case},
we conclude that
\[
\hat{c}(Z_0,B_{Z_0},\bM.;\Sigma_{Z_0})\geq 0
\]
proving $(1)$ of the proposition.
On the other hand, 
if $\hat{c}(X,B,\bM.;\Sigma)=0$, then
$\hat{c}(Z_0,B_{Z_0},\bM.;\Sigma_{Z_0})=0.$
Furthermore, by Lemma~\ref{lem:decomp-contraction} all the divisors contracted by
$X\rightarrow Z_0$ are log canonical centers of $(X,B,\bM.)$.
We conclude that $(X,B,\bM.)$ is Fano type.
Thus, $(X,\lfloor B\rfloor)$ is toric by Proposition~\ref{prop:theorem-FT-case}.\\ 

Now, assume that the contraction $Z\rightarrow C'$ induced by $C_0$ is a Mori fiber space.
In this case, the divisor $K_Z+B_Z-\epsilon_1 B_1 -\epsilon_2 B_2+\bM Z.$
is anti-ample. We conclude that $Z$ is Fano type. Then, the argument in the previous paragraph implies that $X$ is Fano type as well.
Hence, the proposition follows from Proposition~\ref{prop:theorem-FT-case}.
\end{proof} 

\subsection{Reduction to the Fano type case} 
\label{subsec:reduct-FT}
In this subsection, we show the first two statements of Theorem~\ref{thm:gen-log-CY-comp}.
More precisely, we prove the following proposition. 

\begin{proposition}\label{prop:part1-2}
Let $(X,B,\bM.)$ be a generalized log Calabi--Yau pair of dimension $2$.
The following statements hold.
\begin{enumerate}
    \item The inequality $\hat{c}(X,B,\bM.)\geq 0$ holds.
    \item If $\hat{c}(X,B,\bM.)=0$, then $(X,\lfloor B\rfloor)$ is a toric pair.
\end{enumerate}
\end{proposition} 

We will need the following lemma
regarding contraction of curves.

\begin{lemma}\label{lem:curve-contraction} 
Let $(X,B,\bM.)$ be a generalized log canonical pair of dimension $2$.
Let $X\rightarrow Z$ be the contraction of an irreducible curve $C$.
Assume that $(X,B,\bM.)$ is generalized log Calabi--Yau over the image of $C$.
Let $\Sigma$ be a decomposition of the generalized pair. 
Assume the following conditions:
\begin{itemize}
\item Every component of $\Sigma$ intersects $C$ positively,
\item $C$ is not a component of $\Sigma$.
\end{itemize} 
Then, $|\Sigma|\leq 2$. 
\end{lemma} 

\begin{proof}
Assume that $|\Sigma|>2$.
We may assume that $(X,B,\bM.)$ has a log canonical center on $C$. 
Otherwise, we can pull back an effective divisor passing through the image of $C$ on $Z$ and add it to $B$.
Since we are adding a divisor that is pull-back from $Z$ this does not change
the relative Calabi--Yau condition.
Let $\phi\colon Y\rightarrow X$ be a projective birational morphism that extracts a log canonical place $E$ of $(X,B,\bM.)$ whose log canonical center is contained in $C$. 
Let $B_Y$ be the strict transform 
of $B$ on $Y$ and $C_Y$ be the strict transform of $C$ on $Y$.
By the projection formula, we have that
\[
0 > C^2 = \phi^*(C)\cdot C_Y = (C_Y+\alpha E)\cdot C_Y > C_Y^2,
\]
so $C_Y^2$ is negative and $C_Y$ is an extremal curve.
On the other hand, by the log Calabi--Yau structure, we have that 
\[
0 = (K_Y+B_Y+E+\bM Y.) \cdot C_Y \geq   
(K_Y+B_Y+E) \cdot C_Y > 
(K_Y+B_Y) \cdot C_Y.
\]
We conclude that $C_Y$ is an extremal
$(K_Y+B_Y)$-negative curve.
Then, we may contract $C_Y$.
Let $Y\rightarrow X'$ be the contraction of $C_Y$.
Denote by $E'$ the image of $E$ on $X'$.
Note that $X'\rightarrow Z$ is a projective birational morphism.
Let $B'$ be the push-forward of $B_Y$ to $X'$.
Note that $\rho(X/Z)=1$,
so we conclude that 
$\rho(X'/Z)=1$.
In particular, each component of $\Sigma'$ intersects $E'$ positively.
We may replace $(X,B,\bM.)$ with
$(X',B',\bM.)$ and $C$ with $E'$.
We replace $\Sigma$ with the induced decomposition $\Sigma'$ on $(X',B',\bM.)$.
By doing so, we may assume that 
$C$ appears with coefficient one in $B$. Although, it does not appear in $\Sigma$.
By performing adjunction of $(X,B,\bM.)$ to $C$, we obtain a generalized log Calabi--Yau pair
$(\pp^1,B_{\pp^1},\bM \pp^1.)$ 
and a decomposition $\Sigma_{\pp^1}$ 
for which $|\Sigma_{\pp^1}|>2$.
This contradicts Proposition~\ref{prop:theorem-FT-case}.
Indeed, we get
$c(\pp^1,B_{\pp^1},M_{\pp^1})<0$.
\end{proof}

\begin{proof}[Proof of Proposition~\ref{prop:part1-2}]
Assume we know the statement for $\qq$-factorial klt varieties.
Let $(X,B,\bM.)$ be a generalized pair as in the statement. 
Let $(Y,B_Y,\bM.)$ be a generalized $\qq$-factorial dlt modification.
Note that $Y$ is klt and $\qq$-factorial.
As in the proof of Step 1 in the proof of Proposition~\ref{prop:LCY-case-with-full-span}, we may find a decomposition
$\Sigma_Y$ of $(Y,B_Y,\bM.)$ for which 
\[
\hat{c}(Y,B_Y,\bM.)\leq \hat{c}(X,B,\bM.).
\]
By the $\qq$-factorial klt case, we conclude $\hat{c}(X,B,\bM.)\geq 0$.
If $\hat{c}(X,B,\bM.)=0$, then
$\hat{c}(Y,B_Y,\bM.)=0$ and 
by the $\qq$-factorial klt case
we conclude that $(Y,\lfloor B_Y\rfloor)$ is a toric pair.
This implies that $(X,\lfloor B\rfloor)$ is a toric pair.
Thus, from now on, we need to prove the statement assuming that $X$ is $\qq$-factorial and klt.

We proceed by induction on the Picard rank on $X$.
If $\rho(X)=1$, then the statement follows from Proposition~\ref{prop:theorem-FT-case}.
Indeed, first, note that in this case $X$
is $\qq$-factorial.
If $B=0$ and $\bM X.=0$, then 
$\hat{c}(X,B,\bM.)$ is automatically positive.
On the other hand, if either $B\neq 0$
or $\bM X.\neq 0$, then $K_X$ must be anti-ample and hence $X$ is Fano type.
Thus, from now on, we may assume that $\rho(X)\geq 2$ and the statement holds for varieties with lower Picard rank.

We run a $K_X$-MMP. 
If the MMP starts with a Mori fiber space to a curve, 
then the same argument in Step 3 in the proof of Proposition~\ref{prop:LCY-case-with-full-span}
implies that $X$ is Fano type.
Then, we conclude again by Proposition~\ref{prop:theorem-FT-case}.
Thus, we may assume that the first step of the MMP
is a divisorial contraction $X\rightarrow Y$.
Let $C$ be the contracted curve. 
Let $(Y,B_Y,\bM.)$ be the induced log Calabi--Yau pair structure on $Y$.
Let $\Sigma_Y$ be the decomposition on $(Y,B_Y,\bM.)$ 
obtained by pushing-forward the decomposition $\Sigma$ on $X$.
By Lemma~\ref{lem:decomp-contraction},
we have that 
\begin{equation}\label{eq:ineq-comp}
\hat{c}(Y,B_Y,\bM.;\Sigma_Y)
\leq \hat{c}(X,B,\bM.;\Sigma).
\end{equation} 
This shows the first statement.

Now, assume that $\hat{c}(X,B,\bM.;\Sigma)=0$.
Then, $\hat{c}(Y,B_Y,\bM.;\Sigma_Y)=0$
so $(Y,\lfloor B_Y\rfloor)$ is toric by the induction hypothesis.
If $\Sigma$ contains $C$ in its support, then the previous paragraph implies that $C$ is a log canonical center of $(Y,B_Y,\bM.)$.
In this case, $X$ is Fano type, so
the statement follows from Proposition~\ref{prop:theorem-FT-case}.
Thus, we may assume that $C$ does not appear in $\Sigma$.
Furthermore, we know that in this case
$\rho(\Sigma_Y)=\rho(Y)$.\\

\textit{Claim} There exists a projective birational morphism $Y\rightarrow T$ for which either:
\begin{enumerate}
\item The Picard rank of $T$ equals one; or
\item The Picard rank of $T$ is two and both extremal contractions define Mori fiber spaces.
\end{enumerate}
Furthermore, there exists a projective birational morphism
$X\rightarrow X'$ which contract the strict transforms of all the curves in ${\rm Ex}(Y/T)$.\\

\begin{proof}[Proof of the Claim]
Let $y\in Y$ be the image of $C$ on $Y$.
First, assume that $y$ is contained in a prime component $S_Y$ of $\lfloor B_Y\rfloor$.
We run a $(K_Y+B_Y-S_Y+\bM Y.)$-MMP. 
Let $Y=:T_0\rightarrow \dots \rightarrow T_k$ be the divisorial contractions of this MMP and $T_k\rightarrow Z$ be the Mori fiber space. 
Let $X_0:=X$.
Inductively on $i$, we construct a projective birational morphism $X\rightarrow X_i$ which contracts the strict transforms of all the curves in ${\rm Ex}(Y/T_i)$. 
Assume $X_i$ is already constructed.
Let $(X_i,B_i,\bM.)$ be the induced generalized log Calabi--Yau structure.
Let $\phi_i\colon X_i\rightarrow T_i$ be the projective morphism contracting $C_i$ the strict transform of $C$ on $X_i$.
Let $E_i$ be the irreducible exceptional divisor of $T_i\rightarrow T_{i+1}$.
By Lemma~\ref{lem:decomp-contraction}, the curve $E_i$ is a log canonical center of $(T_i,B_{T_i},\bM T_i.)$.
The curve $E_i$ intersects the strict transform of $S_Y$ in $T_i$.
Thus, the dual complex $\mathcal{D}(T_i,B_{T_i},\bM T_i.)$ is positive dimensional.
Hence, the dual complex
$\mathcal{D}(X_i,B_i,\bM X_i.)$
is also positive-dimensional.
Let $F_i$ be the strict transform of $E_i$ in $X_i$.
Note that $F_i$ is extremal. Indeed, by the projection formula, we have that
\[
0>F_i^2=\phi_i^*(E_i)\cdot F_i \geq F_i^2.
\]
By connectedness of generalized log canonical centers, there exists a component $\Gamma_i$ of $\lfloor B_i\rfloor$ that intersects $F_i$ positively (see, e.g.,~\cite[Theorem 1.1]{FS20}).
Indeed, the strict transform of $S_Y$ on $X_i$ and $F_i$ appear in $B_i$ with coefficient one.
Then, $F_i$ is a $(K_{X_i}+B_i-\Gamma_i+\bM X_i.)$-negative curve.
Thus, we may contract $F_i$ to produce the model $X_{i+1}$.

If $\rho(T_k)=1$, then we are in the first case. 
If $\rho(T_k)=2$ and the second extremal contraction of $T_k$ is a Mori fiber space, then we are in the second case.
Assume that $\rho(T_k)=2$ and the second extremal contraction $T_k\rightarrow T_{k+1}$ of $T_k$ is birational. Let $E_k$ be the irreducible exceptional divisor.
Let $S_k$ be the strict transform of $S$ on $T_k$.
If $E_k\neq S_k$ and $E_k\cap S_k\neq \emptyset$, then we may contract $F_k$ by the same argument of the previous paragraph.
If $E_k\neq S_k$ and $E_k\cap S_k=\emptyset$, then $S_k$ and $E_k$ are sections of a Mori fiber space structure $T_k\rightarrow \pp^1$.
There exists a component of $B_{T_k}$
or $\bM T_k.$ that intersects $E_k$ positively.
Since the exceptional locus of
$X_k\rightarrow T_k$ is disjoint from $E_k$, then there exists a component $\Gamma_k$ of $B_k$ or $\bM X_k.$ that intersects $F_k$ positively.
Then, $F_k$ is an extremal curve and
$(K_{X_k}+B_k+\bM X_k.-\epsilon \Gamma_k)$-negative for $\epsilon>0$ small enough. 
Hence, we may contract the curve $F_k$.

From now on, we assume $E_k=S_k$.
Let $\Sigma_k$ be the induced decomposition on $X_k$
and $\Sigma_{T_k}$ be the induced decomposition on $T_k$.
By Lemma~\ref{lem:curve-contraction}, the sum of the coefficients of the components of $\Sigma_k$, different from $\phi^{-1}_*(S_k)$ and intersecting $C_k$ positively, is at most one.
By restricting to a general fiber of $T_k\rightarrow Z$, we observe that  the sum of the coefficients of the components of $\Sigma_{T_k}$ that intersect $E_k$ trivially is at most one.
Since $|\Sigma_k|=4$, we conclude that there is a component of $\Sigma_k$ intersecting $S_k$ positively.
Let $\Gamma_k$ be such a component.
We conclude that $S_k$ is an extremal
$(K_{X_k}+B_k-\Gamma_k+\bM X_k.)$-negative curve, so we may contract it and obtain the model $X_{k+1}=:X'$.

Finally, assume that $y\in Y$ is not contained in any prime component of $\lfloor B_Y\rfloor$.
Let $Y\rightarrow T$ be any morphism to a toric surface satisfying either $(1)$ or $(2)$.
Then, by Lemma~\ref{lem:decomp-contraction}, we conclude that
${\rm Ex}(Y/T)$ is contained in $\lfloor B_Y\rfloor$.
Thus, we conclude that $y$ is not contained in ${\rm Ex}(Y/T)$.
In this case, it is clear that $X\rightarrow X'$ exists.
\end{proof}

We let $(X',B',\bM.)$ be the induced generalized log Calabi--Yau pair on $X'$. 
Let $\Sigma'$ be the decomposition on $(X',B',\bM.)$ induced by Lemma~\ref{lem:decomp-contraction}.
Then, we have that
$\hat{c}(X',B',\bM.;\Sigma')=0$.
By construction, the push-forward $C'$ of $C$ in $X'$
does not appear in the decomposition $\Sigma'$.
Thus, we have a commutative diagram of log Calabi--Yau pairs with decompositions:
\begin{align*}
\xymatrix{
(X,B,\bM.;\Sigma)\ar[r]\ar[d] &
(X',B',\bM.;\Sigma') \ar[d]^-{\psi}\\
(Y,B_Y,\bM.;\Sigma_Y)\ar[r] &
(T,B_T,\bM.;\Sigma_T). 
}
\end{align*} 
Furthermore, all the decompositions have associated orbifold complexity equal to $0$.
It suffices to show that $X'$ is a toric variety. 
Indeed, if $X'$ is a toric variety, then it is Fano type.
As all the divisors extracted by $X\rightarrow X'$
have log discrepancy $0$ with respect to $(X',B',\bM.)$, this implies that $X$ is of Fano type.
We will proceed in two cases, depending on the Picard rank of $T$.\\

\textit{Case 1:} In this case, we assume that the Picard rank of $T$ is one.\\

In this case, we have that $\rho(X')=2$.
If $\rho(\Sigma')=\rho(X')$, then the statement follows
from Proposition~\ref{prop:LCY-case-with-full-span}.
Hence, we may assume that $\rho(\Sigma')=1$.
In particular, all the divisors that appear in the decomposition of $\Sigma'$ are $\rr$-proportional.
Note that $|\Sigma'|=3$.
If a component of $\Sigma'$ intersects $C'$ positively, then we get a contradiction by Lemma~\ref{lem:curve-contraction}.
Hence, we may assume that no component of $\Sigma'$ intersects $C'$.
Thus, $X'\rightarrow T$ is crepant.
So $X'$ must be a toric variety.\\

\textit{Case 2:} In this case, we assume that the Picard rank of $T$ equals $2$ and both extremal contractions correspond to Mori fiber spaces.\\

First, note that some component of $\Sigma'$ must intersect $C'$ positively, otherwise $\psi\colon X'\rightarrow T$ is crepant and $T$ is toric.
Let $\pi_1\colon T\rightarrow C_1$ and $\pi_2\colon T\rightarrow C_2$ be the two Mori fiber space structures of $T$.
Let $\phi_1\colon X'\rightarrow C_1$ and
$\phi_2\colon X'\rightarrow C_2$ be the two corresponding fibrations.
We prove this case in four steps.\\

\textit{Step 1:} In this step, we show that every component of $\Sigma_T$ is $\rr$-proportional to a fiber of either $\pi_1$ or $\pi_2$.\\

We can write $\Sigma_T=\sum_{i=1}^\ell b_iW_i$ where each $b_i$ is a positive real number and each $W_i$ is a Weil divisor.
Here, the $W_i$'s are the components of $\Sigma_T$.
Since $|\Sigma_T|=4$, we have that
\begin{equation}
\label{eq:sum4}
\sum_{i=1}^\ell b_i=4.
\end{equation} 
The cone of effective divisors of $T$ is spanned by $f_1$ and $f_2$ which are reduced fibers of $\phi_1$ and $\phi_2$ respectively.
In particular, we have 
$W_i\sim a_{i,1}f_1+a_{i,2}f_2$ for each $i$, where the $a_{i,j}$'s are non-negative integers.
On the other hand,
we have that $-K_T\sim 2f_1+2f_2$ by~\cite[Theorem 8.2.3]{CLS11}.
Thus, we conclude that 
\[
\sum_{i=1}^{\ell} b_ia_{i,1} =2
\text{ and }
\sum_{i=1}^{\ell} b_ia_{i,2} =2
\]
Hence, the following equality holds
\begin{equation}\label{eq:total-sum}
\sum_{i=1}^{\ell} b_i(a_{i,1}+a_{i,2}) =4.
\end{equation}
Comparing~\eqref{eq:sum4}
and~\eqref{eq:total-sum}, we conclude that $a_{i,1}+a_{i,2}=1$ for each $i$, which means that each $W_i$ is linearly equivalent to either $f_1$ or $f_2$.\\

\textit{Step 2:} In this step, we show that every component of $\Sigma_{\bM X'.}$ is $\rr$-proportional 
to a fiber of either $\phi_1$ or $\phi_2$.\\

By the first step, every component of $\Sigma_T$ is $\rr$-proportional to a fiber of either Mori fiber space. 
We conclude that every component of $\Sigma'$
is $\rr$-linearly equivalent to a divisor of the form
$\lambda \psi_{*}^{-1}\pi_i^*(p) +\mu C'$ 
for $\lambda \geq 0$, $i\in \{1,2\}$, $p\in \pp^1$, and $\mu \in \rr$. 
A component of $\Sigma_{\bM X'.}$ is nef and
$\rr$-linearly equivalent to 
a divisor of the form
\[
\lambda \psi_{*}^{-1}\pi_i^*(p) +\mu C'
\]
for $\lambda \geq 0, i\in \{1,2\}, p\in \pp^1,$
and $\mu \in \rr$.
Set $q=\phi_i(C')$.
If $p\neq q$ and $\mu>0$, then 
\[
(\lambda \psi_{*}^{-1}\pi_i^*(p) +\mu C')\cdot C' = \mu (C')^2 <0, 
\]
leading to a contradiction.
If $p\neq q$ and $\mu<0$, then 
\[
(\lambda \psi_{*}^{-1}\pi_i^*(p) +\mu C')\cdot 
\psi_{*}^{-1}\pi_i^*(q)<0,
\]
leading to a contradiction.
We conclude that if $p\neq q$, then $\mu=0$ and 
the component of $\Sigma_{\bM X'.}$ is $\rr$-proportional to a fiber of $\phi_i$.
If $p=q$, then we can find $\epsilon\in \rr$ for which the divisor 
\begin{equation}\label{eq:nef-div-rel}
\lambda \psi_{*}^{-1}\pi_i^*(p) +\mu C' - \epsilon \phi_i^*(p), 
\end{equation} 
is not supported on the whole fiber over $p$ with respect to $\phi_i$. 
The divisor~\eqref{eq:nef-div-rel} is nef over $C_i$ if and only if it is trivial. 
Hence, we conclude that the component is $\rr$-equivalent to $\epsilon \phi_i^*(p)$. This finishes the proof of the second step.\\

\textit{Step 3:} In this step, we show that some component of $\Sigma'$ is $\rr$-proportional to a fiber of $\phi_i$ for each $i\in \{1,2\}$.\\ 

We show that some component of $\Sigma'$ is $\rr$-proportional to a fiber of $\phi_1$. 
The other statement is analogous.
If a component of $\Sigma_{B_T}$, that is vertical over $C_1$, does not contain $\psi(C')$ in its support, then we are done.
Thus, we may assume that every component of $\Sigma_{B_T}$, that is vertical over $C_1$, contains $\pi(C')$ on its support. 
In particular, the sum of the coefficients of the components of $\Sigma_{B_T}$, that are vertical over $C_1$, is at most $1$.
This follows from the log canonical condition of the pair $(T,B_T)$.
Thus, by the second step, there is a component of $\Sigma_{\bM X'.}$ that is $\rr$-proportional to a fiber of $\phi_1$. 
Otherwise, the intersection of $K_{X'}$ with a general fiber of $\phi_2$ is larger than $-2$, leading to a contradiction.\\ 

\textit{Step 4:} In this step, we finish the proof of the second case.\\ 

Recall that some component of $\Sigma'$ intersects $C'$ positively.
By the third step, 
$\Sigma'$ contains at least two further components;
one that is $\rr$-proportional to a fiber of $\phi_1$
and one that is $\rr$-proportional to a fiber of $\phi_2$.
We conclude that $\rho(\Sigma')=3$.
By Proposition~\ref{prop:LCY-case-with-full-span}, we have that $X'$ is toric. 
\end{proof}

\subsection{Toric case}
\label{subsec:toric case}
In this subsection, we finish the proof of the main theorem.
In order to do so, it suffices to show $(3)$ and $(4)$ of Theorem~\ref{thm:gen-log-CY-comp}.
We will use the following two lemmas from toric geometry.

\begin{lemma}\label{lem:toric-rho=2}
Let $X$ be a projective toric surface with $\rho(X)=2$.
Assume that $X$ is not a product of projective lines.
There exists a projective birational morphism
$Y\rightarrow X$ extracting at most one divisor
$E$ with $a_E(X)\in (0,1]$
and a projective birational contraction 
$Y\rightarrow Z$ onto a variety of Picard rank $1$.
\end{lemma}

\begin{proof}
If $X$ itself admits a divisorial contraction, we simply take $X=Y$ and $X\rightarrow Z$ such divisorial contraction.
Hence, assume that $X$ admits no divisorial contractions.
This means that both extremal rays of the cone of curves of $X$ define Mori fiber spaces.
If $\Sigma$ is the fan defining $X$, then
$\Sigma(1)=\{v_1,v_2,-v_1,-v_2\}$. 
If $X$ is smooth, then $X$ is isomorphic to $\pp^1\times \pp^1$.
Thus, we may assume that $X$ is not smooth.
Let $Y\rightarrow X$ be a divisorial contraction that extracts a divisor with log discrepancy in $(0,1]$. 
Then, the exceptional divisor of $Y\rightarrow X$ corresponds to a ray subdivision of the fan $\Sigma$.
Without loss of generality, we may assume that 
the fan of $Y$ is obtained by adding a ray $v_3$ in the relative interior of $\langle v_1,v_2\rangle$. 
Then, we may contract the divisors corresponding to $v_1$ and $v_2$ and obtain a toric variety $Z$ of Picard rank $1$ whose fan corresponds to $\{-v_1,-v_2,v_3\}$. 
\end{proof}

\begin{lemma}\label{lem:toric-rho=1-non-canonical}
Let $X$ be a projective toric surface
of Picard rank one.
Assume that $X$ is not canonical.
Assume that $X$ is not isomorphic to $F_n$.
Then, there exists a projective birational morphism $Y\rightarrow X$ extracting a divisor with $a_E(X)\in (0,1)$ and a projective birational morphism $Y\rightarrow Z$ onto a variety of Picard rank $1$ that does not contract $E$. 
\end{lemma}

\begin{proof}
Let $\Sigma(1)=\{v_1,v_2,v_3\}$ be the 
rays of the fan of $X$.
Let $Y_0\rightarrow X$ be a projective birational morphism extracting a divisor $F$ with $a_F(X)\in (0,1)$.
Let $v_4$ be the ray corresponding to $F$ which we may assume to be contained in the relative interior of $\langle v_2,v_3\rangle$.
If $v_4\neq -v_1$, then we can contract a prime divisor corresponding to either $v_2$ or $v_3$. 
Thus, we may assume that $v_4=-v_1$.
Without loss of generality, we may assume that 
$v_1=(-1,0)$.
If the cones $\langle v_2,v_4\rangle$ 
or $\langle v_3,v_4\rangle$ are not smooth, then we can extract a second divisor $G$ with
$a_G(X)\in (0,1)$ whose corresponding ray $v_5$ belongs to either of these cones.
If $v_5\in {\rm relint}(\langle v_2,v_4\rangle)$, then we may contract $v_2$ and $v_4$.
Otherwise, we may contract $v_3$ and $v_4$.
In both cases, we obtain a toric variety of Picard rank one. 
Finally, we may assume that
both cones $\langle v_2,v_4\rangle$ 
and $\langle v_3,v_4\rangle$ are smooth.
Thus, we reduced to the case in which
$v_1=(-1,0),v_4=(1,0),v_2=(m,1)$, and $v_3=(n,-1)$ for some integers $m$ and $n$.
In this case, $X$ is isomorphic to $F_{m+n}$.
\end{proof}

\begin{proof}[Proof of Theorem~\ref{thm:gen-log-CY-comp}]
Part $(1)$ and $(2)$ follow from Proposition~\ref{prop:part1-2}.
We aim to prove $(3)$ and $(4)$.
We will proceed in four cases, depending 
on the Picard rank and the singularities of $X$.
Let $(X,\bM.)$ be a generalized log Calabi--Yau pair
for which $X$ is a projective toric surface. 
We are assuming that $\hat{c}(X,\bM.)=0$.
Let $\Sigma_{\bM.}$ be a decomposition that computes the orbifold complexity.
We may assume that $\rho(\Sigma)=\rho(X)$.\\

\noindent\textit{Case 1:} In this case, we assume that $\rho(X)\geq 3$.\\

If $\rho(X)\geq 3$, then there exists a non-trivial projective birational contraction $X\rightarrow Y$.
Then, the generalized pair
$(Y,\bM.)$ is a generalized log Calabi--Yau pair with negative complexity.
This leads to a contradiction.\\

\noindent\textit{Case 2:} In this case, we assume that $\rho(X)=2$.\\

Assume that $X$ is not isomorphic to $\pp^1\times \pp^1$.
By Lemma~\ref{lem:toric-rho=2}, there exists a projective birational morphism 
$\pi\colon Y\rightarrow X$ extracting at most one divisor $E$ with $a_E(X)\in (0,1]$ and a projective birational contraction $\phi \colon Y\rightarrow Z$ onto a variety of Picard rank $1$.
Write
\[
\pi^*(K_X+\bM X.)=
K_Y+(1-a_E(X))E+\bM Y. 
\]
Then, the pair
$(Y,(1-a_E(X))E,\bM.)$ a generalized log Calabi--Yau pair.
Let $E_Z$ be the push-forward of $E$ in $Z$. 
Then, the generalized pair
\[
(Z,(1-a_E(X))E_Z,\bM.) 
\]
is generalized log Calabi--Yau
with negative orbifold complexity
as $\rho(Z)=1$.
We conclude that if
$\rho(X)=2$, then $X\simeq \pp^1\times\pp^1$.\\

\noindent\textit{Case 3:} In this case, we assume that $\rho(X)=1$ and $X$ is not canonical.\\

We proceed in two subcases depending whether $X$ is isomorphic to $F_n$ or not.\\

\noindent\textit{Case 3.1:}
In this case, we assume that $\rho(X)=1$ and $X$ is not isomorphic to $F_n$.\\ 

By Lemma~\ref{lem:toric-rho=1-non-canonical}, there exists:
\begin{itemize}
\item a projective birational morphism $Y\rightarrow X$ extracting a divisor $E$ with $a_E(X)\in (0,1)$, and 
\item a projective birational morphism $Y\rightarrow Z$ onto a variety $Z$ of Picard rank $1$ that does not contract $E$.
\end{itemize}
Let $(Y,(1-a_E(X))E,\bM.)$
be the log pull-back of
$(X,\bM.)$ to $Y$.
Note that $1-a_E(X)>0$.
Let $E_Z$ be the push-forward of $E$ to $Z$.
Then, we obtain a generalized log Calabi--Yau pair
$(Z,(1-a_E(X))E_Z,\bM.)$ of negative orbifold complexity.
This leads to a contradiction.\\ 

\noindent\textit{Case 3.2:}
In this case, we assume that $X\simeq F_n$.\\ 

Let $(\Sigma_n,\alpha C_0,\bM.)$ be the log pull-back to $\Sigma_n$. 
Here, $C_0$ denotes the $(-n)$-curve.
Let $f$ be a fiber of 
the Mori fiber space structure
$\Sigma\rightarrow \pp^1$.
The cone of effective curves
of $\Sigma_n$ is spanned by $C_0$ and $f$.
Recall that 
\[
-K_{\Sigma_n}\sim 2C_0 + (n+2)f.
\]
A curve in $\Sigma_n$ is nef if and only if it is linearly equivalent to
$\lambda f$ with some positive integer $\lambda$
or it has the form
$aC_0+bf$ with $b\geq na$.
Thus, we may write the push-forward of $\Sigma_{\bM.}$ on $X$ as 
\[
\sum_{i=1}^k\lambda_i(a_iC_0+b_if)
+
\sum_{i=1}^\ell\mu_i(c_if), 
\]
where the $a_i,b_i$, and $c_i$'s are positive integers,
$b_i\geq na_i$ for each $i$, and 
the $\lambda_i$ and $\mu_i$ are positive real numbers.
Thus, we conclude that the following equalities hold:
\begin{align}
\sum_{i=1}^k \lambda_i +
\sum_{i=1}^\ell \mu_i &=3\\
\sum_{i=1}^k \lambda_ia_i+\alpha &=2\\
\sum_{i=1}^k \lambda_ib_i+\sum_{i=1}^\ell \mu_ic_i &= n+2.
\end{align}
By $(3.10)$, $(3.11)$, 
the fact that $c_i\geq 1$,
and the fact that $b_i\geq na_i$, 
we conclude that 
\begin{align}
2-n+\alpha n \geq \sum_{i=1}^\ell \mu_i.
\end{align}
On the other hand, we have that
\begin{align}
2-\alpha \geq \sum_{i=1}^k \lambda_i.
\end{align}
Plugging $(3.12)$ and $(3.13)$ in $(3.9)$, we conclude that 
$2-n+\alpha n+2-\alpha \geq 3$
which implies
$1-\alpha \geq n(1-\alpha)$.
This can only happen if $\alpha=1$.
Thus, $(\Sigma_n,C_0;\bM.)$ is not generalized klt. 
Hence, $(F_n,\bM.)$ is not generalized klt.
This proves $(3)$ and $(4)$ 
in the case that $X$ is not canonical and $\rho(X)=1$.\\

\noindent\textit{Case 4:} In this case, we assume that
$\rho(X)=1$ and $X$ is canonical.\\ 

There are only $5$ canonical toric surfaces of Picard rank one.
This follows from the classification
of toric canonical Fano surfaces in~\cite[Proposition 4.1]{Nil05}.
One of these cases is $\pp^2$.
The case $F_2$ is already proved in the previous case. 
Thus, we proceed in three cases depending on the fan $\Sigma$ of $X$.
In the following three cases, we show that there is no generalized log Calabi--Yau pair structure $(X,\bM.)$ by deriving a contradiction.\\

\noindent\textit{Case 4.1:} In this case, we assume that $\Sigma=\{(-2,1),(1,1),(1,-2) \}$.\\

The minimal resolution $Y$ of $X$ has associated
fan
\[
\Sigma_Y=
\{
(-2,1),
(-1,1), 
(0,1),
(1,1),
(1,0),
(1,-1),
(1,-2),
(0,-1),
(-1,0)
\}.
\]
Consider $Z_1$ be the toric surface with associated fan
$\Sigma_1=\{(-1,1),(1,0),(0,-1)\}$
and $Z_2$ be the toric surface with associated fan 
$\Sigma_2=\{(-1,0),(0,1),(1,-1)\}$.
Then, we have that $Z_1\simeq Z_2\simeq \pp^2$.
If $E$ is any exceptional of $Y\rightarrow X$,
then $E$ is not contracted either on $Z_1$ or $Z_2$.
Assume that $a_E(X,\bM.)<1$ for some $E$ on $Y$.
Assume that the push-forward $E_{Z_1}$ of $E$ on $Z_1$
is a divisor.
Then, the generalized log Calabi--Yau pair
$(Z_1,(1-a_E(X,\bM.))E_{Z_1},\bM.)$ has negative complexity.
This leads to a contradiction.
We conclude that $a_E(X,\bM.)=1$ for every 
$E\subset Y$ prime exceptional over $X$.
Since $a_E(X)=1$, this implies that every component 
of $\bM Y.$ intersects $E$ trivially.
Since $Y$ is smooth, we conclude that each component 
of $\bM X.$ is Cartier. 
Hence, each component of $\bM X.$ is ample Cartier.
This contradicts Kobayashi-Ochiai Theorem for surfaces.\\

\noindent\textit{Case 4.2:}
In this case, we assume that $\Sigma=\{(0,1),(-2,-1),(2,-1)\}$.\\

The minimal resolution $Y$ has associated fan
\[
\Sigma_Y=\{
(0,1),(1,0),(2,-1),(1,-1),(0,-1),(-1,-1),(-2,-1),(-1,0)
\}. 
\]
Consider the toric surfaces $Z_1$ and $Z_2$ given by the fans
$\{(-1,0),(1,-1),(0,1)\}$
and $\{(-1,-1),(1,0),(0,1)\}$.
Then, if $a_E(X,\bM.)<1$ for some $E\in \{(-1,0),(1,0),(-1,-1),(1,-1)\}$, then we can consider the log pull-back of $(X,\bM.)$ to $Y$
and push it forward to either $Z_1$ or $Z_2$ to derive a contradiction.
Indeed, by doing so, we obtain a generalized log Calabi--Yau pair with negative complexity.
We conclude that
$a_E(X,\bM.)=1$ for every $E\in \{(-1,0),(1,0),(-1,-1),(1,-1)\}$
and $a_E(X,\bM.)<1$ for $E=(0,-1)$. 
Let $Z$ be the projective toric surface with associated fan
$\Sigma_Z=\{(-1,0),(0,1),(1,0),(0,-1)\}$.
Then, each component of $\bM Z.$ is nef Cartier.
We write $C_0$ for the exceptional curve of $Z\rightarrow X$.
Let $(Z,\alpha C,\bM.)$ be the pull-back of $(X,\bM.)$ to $Z$.
Note that $Z$ admits a Mori fiber space structure
$Z\rightarrow \pp^1$ for which $C_0$ is a section.
We let $f$ be a fiber of this Mori fiber space. 
Then, the cone of curves of $Z$ is spanned by $C_0$ and $f$. 
A divisor is nef Cartier on $Z$ if and only if 
it is linearly equivalent to either 
$2cf$ for some positive integer $c$
or $2aC_0+2bf$ for some positive integers $a$ and $b$ with $b\geq 2a$.
Then, we can write the push-forward of the decomposition $\Sigma_{\bM.}$ on $Z$ as follows:
\[
\Sigma_{\bM Z.}=
\sum_{i=1}^k \lambda_i(2a_iC_0+2b_if)+
\sum_{i=1}^\ell \mu_i(2c_if).
\]
Hence, we conclude that the following equations hold:
\begin{align}
\sum_{i=1}^k \lambda_i +
\sum_{i=1}^\ell \mu_i &=3\\
\sum_{i=1}^k 2\lambda_ia_i+ \alpha &=2\\
\sum_{i=1}^k 2\lambda_ib_i+\sum_{i=1}^\ell 2\mu_ic_i &= 4.
\end{align}
The first and last equalities lead to a contradiction.\\ 

\noindent\textit{Case 4.3:} 
In this case, we assume that $\Sigma=\{(-1,0),(0,1),(3,-2)\}$.\\ 

The minimal resolution $Y$ has associated fan
\[
\Sigma_Y=\{
(0,1),(1,0),(2,-1),(3,-2),(1,-1),(-1,0)
\}. 
\]
Consider the toric surfaces $Z_1$ and $Z_2$ given by the fans
$\{(-1,0),(0,1),(2,-1)\}$
and $\{(-1,0),(0,1),(1,-1)\}$.
Then, if $a_E(X,\bM.)<1$ for some $E\in \{(1,-1),(1,0)\}$, then we can take the log pull-back of $(X,\bM.)$ to $Y$
and push it forward to either $Z_1$ or $Z_2$ to derive a contradiction.
Indeed, in this case, we obtain a generalized log Calabi--Yau pair of negative orbifold complexity.
We conclude that
$a_E(X,\bM.)=1$ for every $E\in \{(1,0),(1,-1)\}$
and $a_E(X,\bM.)<1$ for $E$
the divisor corresponding to the ray $(2,-1)$. 
Let $Z$ be the projective toric surface with associated fan
$\Sigma_Z=\{(-1,0),(0,1),(3,-2),(2,-1)\}$.
Then, each component of $\bM Z.$ is nef Cartier.
We write $C_0$ for the exceptional curve of $Z\rightarrow X$.
Let $(Z,\alpha C,\bM.)$ be the pull-back of $(X,\bM.)$ to $Z$.
Then, the cone of curves of $Z$ is generated by $f$ and $C_0$, where $f$ is a fiber of the Mori fiber space structure $Z\rightarrow \pp^1$ that passes through an $A_1$ singularity.
A divisor is nef Cartier on $Z$ if and only if 
it is linearly equivalent to either
$2cf$ for some positive integer $c$
or $2aC_0+2bf$ for some positive integers $a$ and $b$ with $b\geq 3a$.
Then, we can write the push-forward of the decomposition $\Sigma_{\bM.}$ on $Z$ as follows:
\[
\Sigma_{\bM Z.}=
\sum_{i=1}^k \lambda_i(2a_iC_0+2b_if)+
\sum_{i=1}^\ell \mu_i(2c_if).
\]
Hence, we conclude that the following equations hold:
\begin{align}
\sum_{i=1}^k \lambda_i +
\sum_{i=1}^\ell \mu_i &=3\\
\sum_{i=1}^k 2\lambda_ia_i+ \alpha &=2\\
\sum_{i=1}^k 2\lambda_ib_i+\sum_{i=1}^\ell 2\mu_ic_i &= 6.
\end{align}
The first and last equalities lead to a contradiction.\\ 

Putting cases $4.1,4.2,$ and $4.3$ together,
we conclude the following. If $\hat{c}(X,\bM.)=1$
and $X$ is a canonical toric surface with $\rho(X)=1$,
then $X$ is either $F_2$ or $\pp^2$.
This concludes the proof.
\end{proof}

Now, we turn to prove the corollary.

\begin{proof}[Proof of Corollary~\ref{introcor:descend-p2}]
The fact that $\sum_{i=1}^k\lambda_i\leq 3$ follows from Theorem~\ref{thm:gen-log-CY-comp}.(1). 
If $\sum_{i=1}^k\lambda_i=3$, then
$B=0$ by Theorem~\ref{thm:gen-log-CY-comp}.(1).
By Theorem~\ref{thm:gen-log-CY-comp}.(3), 
we conclude that $X$ is either isomorphic to
$\pp^2,\pp^1\times\pp^1$ or $X\simeq F_n$.
In the two latter cases, there is a component of $\bM.$ which is not big in the model where it descends.
This follows from the proof of Theorem~\ref{thm:gen-log-CY-comp}.
We conclude that $X\simeq \pp^2$.
By~\cite[Theorem 1.8]{Fil18b}, we conclude that $\bM.$ descends on $\pp^2$.
Indeed, in this case, we know that
each component $\bM \pp^2.$ is a line, 
so its self-intersection is $1$.
\end{proof} 

\section{Generalized complexity for threefold singularities}

In this section, we prove our main theorem regarding 
threefold singularities.
First, we prove the following lemma.

\begin{lemma}\label{lem:lifting-div}
Let $(X,B,\bM.)$ be a generalized pair of dimension $3$ with a decomposition $\Sigma$ for which $\hat{c}(X,B,\bM.;\Sigma)=0$.
Let $S$ be a projective prime component of $\lfloor B\rfloor$ which is normal.
Assume that the restriction of every component of $\Sigma$ to $S$ is not numerically trivial.
Assume that $S$ appears on $\Sigma$ with coefficient $1$.
Then, the restriction of every irreducible component
of $B$ to $S$ is irreducible.
\end{lemma}

\begin{proof}
By Proposition~\ref{prop:gen-comp-under-adj}, we may find a decomposition
$\Sigma_S$ of $(S,B_S,\bM S.)$ with 
$\hat{c}(S,B_S,\bM S.;\Sigma_S)=0$.
By Theorem~\ref{prop:gen-comp-under-adj}, we conclude
that $\rho(\Sigma_S)=\rho(S)$.
Then, the statement follows from Remark~\ref{rem:irreducible-comp}.
\end{proof}

\begin{proof}[Proof of Theorem~\ref{introthm:3-fold-sing}]
Let $(X;x)$ be a klt $3$-fold singularity.
Let $(X,B,\bM.;x)$ be a generalized log canonical singularity.
Let $\Sigma$ be a decomposition
of this generalized pair
for which $c(X,B,\bM.;x)<0$.
Without loss of generality, 
we may assume that $(X,B,\bM.;x)$ is generalized log
canonical
and $\{x\}$ is a generalized log canonical center.
Let $Y\rightarrow X$ be a plt blow-up for $(X;x)$
for which the exceptional divisor
is a log canonical place of $(X,B,\bM.;x)$.
We let $E$ be the exceptional divisor
of $Y\rightarrow X$.
Let $(Y,B_Y+E,\bM.)$ be the log pull-back of
$(X,B,\bM.;x)$ to $Y$.
Let $\Sigma_Y$ be the induced decomposition
on $(Y,B_Y+E,\bM.)$ obtained by adding
with coefficient $1$ the divisor $E$ to the boundary decomposition.
Then, we have that
\[
c(Y,B_Y+E,\bM.;\Sigma_Y)
\leq 
c(X,B,\bM.;\Sigma).
\]
Let $(E,B_E,\bM E.)$ be the generalized pair
obtained by adjunction
of $(Y,B_Y+E,\bM.)$ to $E$.
By Proposition~\ref{prop:gen-comp-under-adj}, 
we know that there exists a
decomposition
$\Sigma_E$ of $(E,B_E,\bM E.)$ for which
\[
\hat{c}(E,B_E,\bM E.)
\leq
c(Y,B_Y+E,\bM.;\Sigma_Y).
\] 
Note that $(E,B_E,\bM.E)$
is a generalized log Calabi--Yau surface.
By Theorem~\ref{thm:gen-log-CY-comp},
we know that
\[
\hat{c}(E,B_E,\bM E.)\geq 0.
\]
This leads to a contradiction.
We conclude that
$c(X,B,\bM.;\Sigma)\geq 0$.

Now, assume that $c(X,B,\bM.;\Sigma)=0$.
In the previous construction, we obtain
that $\hat{c}(E,B_E,\bM E.)=0$.
By Theorem~\ref{thm:gen-log-CY-comp}, 
we conclude that $E$ is a projective toric variety.
Let $\Sigma_{\bM E.}=\sum_{i=1}^k \lambda_i M_i$ 
be the decomposition of the moduli divisor $\bM E.$.
We write $M_{E,i}$ for the push-forward of $M_i$ on $E$.
By Step 3 in the proof of Proposition~\ref{prop:theorem-FT-case}, we conclude that for each 
$i\in \{1,\dots,k\}$ there exists a boundary divisor
$\Gamma_i\sim M_{E,i}$ satisfying the following.
The pair $(E,B_E+\sum_{i=1}^k \lambda_i \Gamma_i)$ is log Calabi--Yau.
Furthermore,
we have that
$\hat{c}(E,B_E+\sum_{i=1}^k \lambda_i\Gamma_i)=0$.
By~\cite[Theorem 4.5]{MS21}, every prime torus invariant divisor
of $E$ is either a component of $B_E$
or a component of $\sum_{i=1}^k \lambda_i\Gamma_i$.
For a prime divisor $F$ of $E$, we denote by $n_F$ the Cartier index of $E$ at $F$.
By Lemma~\ref{lem:lifting-div}, we conclude that for every prime torus invariant prime divisor $F$ of $E$, there either exists:
\begin{itemize}
    \item a component $B_{Y,i}$ of $B_Y$ for which
    $B_{Y,i}|_E = \frac{F}{n_F}$, or 
    \item a component ${\bM Y,i.}$ of $\bM Y.$ for which
    ${\bM Y,i.}|_E\sim \frac{F}{n_F}$.
\end{itemize}
Note that $(\bM Y,i.-E)-K_Y$ is ample over $X$. 
Thus, by Kawamata-Viehweg vanishing, we conclude
that there exists $0\leq G_i \sim \bM Y,i.$ for which
$G_i|_E = \frac{F}{n_F}$.
By~\cite[Lemma 5.1]{MS21}, every divisor $F$ with $n_F>1$ is torus invariant. 

Up to reordering the divisors, 
we may assume that the restrictions of 
$B_{Y,1},\dots,B_{Y,s}$
and 
$G_1,\dots,G_r$ 
are all the prime torus invariant divisors of $E$.
By~\cite[Theorem 4.5]{MS21}, we may assume that each
prime component of $\lfloor B\rfloor$
appears among the divisors $B_{Y,1},\dots,B_{Y,s}$.
By inversion of adjunction,
we conclude that the pair
\begin{equation}\label{eq:lcy-pair}  
\left(
Y,\sum_{i=1}^s B_{Y,i} +\sum_{i=1}^r G_i +E
\right)
\end{equation} 
is log Calabi--Yau over $X$.
Furthermore, the components
$B_{Y,1},\dots, B_{Y,s}, G_1,\dots,G_r$
span a space of dimension at most $r+s-n-1$.
Thus, the log Calabi--Yau pair~\ref{eq:lcy-pair}
has complexity $0$.
Let $(X,\sum_{i=1}^s B_{X,i}+\sum_{i=1}^r G_{X,i};x)$
be the push-forward of~\eqref{eq:lcy-pair} to $X$.
Then, this log pair has complexity $0$
and $\lfloor B\rfloor \subset \supp(\sum_{i=1}^s B_{X,i})$.
By~\cite[Theorem 1]{MS21}, we conclude that $(X,\lfloor B\rfloor)$ is formally toric on a neighborhood of $x$.
\end{proof}

\section{A local version of Kobayashi-Ochiai Theorem}

In this section, we prove a local version of Kobayashi-Ochiai Theorem.
First, we prove a version of Kobayashi-Ochiai theorem
for generalized pairs.

\begin{proof}[Proof of Theorem~\ref{introthm:KO-global}]
Assume that $\sum_{i=1}^k \lambda_i \geq n+1$.
First, we show that the divisor
$K_{X'}+\bM X'.$ is a nef divisor.
First, note that $\bM X'.$ is an ample divisor, 
so every $(K_{X'}+\bM X'.)$-negative curve must be
$K_{X'}$-negative.
By the cone Theorem, every $K_{X'}$-negative
extremal ray is generated by the class of an integral curve 
$C$ which satisfies $K_{X'}\cdot C \geq -(n+1)$.
Thus, we conclude that $(K_{X'}+\bM X'.)\cdot C\geq 0$.
Thus, if $\sum_{i=1}^k \lambda_i \geq n+1$,
then $(K_{X'}+\bM X'.)$ is nef
and if $\sum_{i=1}^k \lambda_i > n+1$,
then $(K_{X'}+\bM X'.)$ is ample.

Let $\pi\colon X'\rightarrow X$ be the associated
projective birational morphism.
Write
\[
\pi^*(K_X+B+\bM X.)= K_{X'}+B'+\bM X'.
\]
By the negativity lemma, we conclude that $B'$ is an effective divisor.
Then, the statement follows from~\cite[Corollary 1.3]{FM21}. 
\end{proof}

\begin{proof}[Proof of Theorem~\ref{introthm:KO-local}]
The first statement follows from Theorem~\ref{introthm:3-fold-sing}.
Hence, we may assume that $\sum_{i=1}^k \lambda_i=3$.
Since $X$ is $\qq$-factorial, 
the exceptional locus of
$\pi\colon X'\rightarrow X$ is divisorial.
Let $C$ be a $K_{X'}$-negative extremal curve over $X$.
Let $X'\rightarrow Y$ be the induced contraction.
Since $X'$ is smooth, this contraction
must be birational.
By~\cite[Theorem 1.32]{KM98}, we know that its exceptional locus $E\subset X'$ is smooth.
By adjunction to $E$, we conclude that 
$K_{X'}+\bM X'.$ intersects $C$ positively.
Indeed, $K_{X'}+E+\bM X'.$ is $C$-positive
and $E\cdot C<0$. 
We deduce that $K_{X'}+\bM X'.$ is ample over $X$.
By the negativity lemma, we conclude that
\[
\pi^*(K_X+B+\bM X.)= K_{X'}+B'+\bM X'.,
\]
satisfies that $B'$ is effective.
Furthermore, the divisor $B'$ must contain every exceptional divisor of $X'\rightarrow X$ on its support. 
Let $\alpha$ be the coefficient of $E$ on $B'$.
We know that the curve
$C$ intersects $K_{X'}+E+\bM X'.$ non-negatively.
Hence, it must intersect 
\[
K_{X'}+E+\bM X'. \sim_{\qq} (1-\alpha)E - (B'-\alpha E).
\]
The curve $C$ intersects $(B'-\alpha E)$ non-negatively and $E$ negatively.
We conclude that $\alpha=1$. 
Thus, $E$ appears in $B'$ with coefficient one. 
Hence, the generalized pair
\[
K_E+B_E+\bM E. \sim (K_X+B'+\bM X'.)|_E,
\]
is generalized log Calabi--Yau.
By Theorem~\ref{introthm:KO-global}, 
we conclude that the following conditions hold:
\begin{itemize}
\item $E$ is isomorphic to $\pp^2$, 
\item the divisor $B_E$ is trivial, and 
\item the generalized pair $(E,\bM E.)$ is generalized terminal.
\end{itemize} 
In particular, no component of $B'_{>0}$ (different than $E$) intersects $E$. 
This implies that $E={\rm Ex}(X'\rightarrow X)$.
Hence, we have that $X'\rightarrow X$ is a resolution of singularities
and its exceptional $E\simeq \pp^2$.
This proves the statement. 
\end{proof}

\section{Examples and Questions}

In this section, we collect some examples and questions.
We start showing that the conditions
on the local Kobayashi-Ochiai theorem are optimal.

\begin{example}
\label{ex:big-nef}
{\em 
In this example, we show that the statement of Theorem~\ref{introthm:KO-local}
fails in dimension $2$ if we drop
the ampleness condition. 
Let $(T;t)$ be any toric surface singularity.
Let $B_1$ and $B_2$ be the two torus invariant divisors
through $t$.
Let $Y\rightarrow T$ be a minimal resolution. 
Write $B_{Y,1}$ and $B_{Y,2}$ be the strict transforms
of $B_1$ and $B_2$, respectively.
Set $\bM Y.:= B_{Y,1}+B_{Y,2}$
be a b-nef divisor.
Then, we have that
$(T,\bM .; t)$ is generalized log canonical.
Furthermore, by construction,
each component of $\bM.$ is big and nef on the smooth model where they descend.
However, $(T;t)$ is not the cone over a rational curve.

Now, we turn to show that the statement
of Theorem~\ref{introthm:KO-local}
fails in dimension $3$ if we drop the ampleness condition. 
Let $(T;t)$ be a toric $3$-fold singularity
such that 
$T\simeq \mathbb{A}^1 \times X$
where $X$ is a toric surface. 
Let $B_1,B_2$, and $B_3$ be the torus invariant divisors of $T$.
Let $Y\rightarrow T$ be a minimal resolution of $(T;t)$.
Then, the strict transforms
$B_{T,1},B_{T,2}$, and $B_{T,3}$
are nef and big over $T$.
We set $\bM Y.:=\sum_{i=1}^3 B_{T,i}$. 
Hence, the pair $(T,\bM.;t)$ is generalized log canonical,
$|\bM.|=3$.
However, $T$ is not a cone over $\pp^2$.
}
\end{example} 

The following is the simplest generalized log Calabi--Yau structure $(F_n,\bM.)$ on $F_n$
with generalized complexity $0$.
Recall that, by Theorem~\ref{thm:gen-log-CY-comp}, this pair must be generalized log canonical
but not generalized klt.

\begin{example}
\label{ex:Fn}
{\em 
Consider the Hirzebruch surface 
$\Sigma_n$
with its Mori fiber space
structure
$\Sigma_n\rightarrow \pp^1$.
Let $S_0$ be the section with self-intersection $-n$.
Let $S_1$ be the section with self-intersection $n$.
Let $F_0$ and $F_1$ be the two torus invariant fibers. 
Consider the boundary divisor
\[
B_{\Sigma_n}:=S_0.
\]
Consider the b-nef divisor defined by
\[
\bM.:= F_0+F_1+S_1.
\]
Note that each of the divisors
$F_0,F_1$ and $S_1$ are nef Cartier. 
Furthermore, the divisor $S_1$ is big and nef.
However, the divisors
$F_0$ and $F_1$ are not big.
Let $(F_n,\bM.)$ be the induced generalized log Calabi--Yau pair structure induced on $F_n$.
Then, we have that 
$|\bM.|=3$, so the generalized complexity is $0$.
}
\end{example}

The following example shows
that in Theorem~\ref{introthm:KO-global}, if we weaken the ampleness hypothesis, 
then the moduli divisor may not descend on $X$ itself.

\begin{example}
\label{ex:not-descend-on-P^2}
{\em
In this example, we show that 
even if $X\simeq \pp^2$
and $(X,\bM.)$ has generalized complexity $0$, then $\bM.$ may not descend on $X$ itself.
This shows that the conditions
on Corollary~\ref{introcor:descend-p2}
are optimal.

Consider $X=\pp^2$ and $\pi\colon Y\rightarrow X$ be the blow-up of $X$ at a point $x$.
Let $E$ be the exceptional divisor.
Let $L$ be the strict transform of a line passing through $x$.
Let $L_1$ and $L_2$ be two lines that do not pass through $x$.
We identify $L_1$ and $L_2$ with their strict transforms on $Y$.
Consider the divisors
$L_1,L_2,L$.
Note that these three divisors are nef Cartier divisors on $Y$.
However, the divisor $L$ is not big.
We consider the b-nef divisor
$\bM.:=L_1+L_2+L$.
Note that $L$ does not descend on $X$.
Indeed, we have that 
$\pi^*\pi_*L=L+E$. 
The generalized pair
$(X,\bM.)$ is generalized log Calabi--Yau and generalized klt.
}
\end{example}

We finish the article with two questions that can lead to further research.

In the statement of Conjecture~\ref{conj:gen-complexity}, it is expected that $(X,\bM.)$ generalized klt log Calabi--Yau of generalized complexity $0$ implies that $X$ is a product of projective spaces.
However, the statement of Theorem~\ref{introthm:surface} shows that this statement already fails if we replace gklt with glc. 
At any rate, something interesting still happens in the $F_n$ example.
There exists a generalized dlt modification of $(F_n,\bM.)$ which is again toric and on which $\bM.$ descends.
However, we are not certain about which toric varieties admit generalized complexity $0$ structures. The following seems like a natural question: 

\begin{question}
{\em 
Can we classify projective toric $3$-folds $T$ for which $(T,\bM.)$ is generalized log Calabi--Yau with complexity $0$?
}
\end{question}

Conjecture~\ref{conj:gen-complexity} connects two important theorems in the literature: Kobayashi-Ochiai and 
the characterization of toric varieties.
There is a third conjecture in the literature, the generalized Mukai conjecture, which seems very similar to the aforementioned theorems.
The generalized Mukai conjecture asserts that for a $d$-dimensional Fano variety, we have that:
\[
\rho(X)(i_X-1)\leq d,
\]
where $i_X$ is the index of the Fano variety.
This invariant is defined as the minimum anti-canonical degree of a curve on $X$.
Furthermore, it is expected that the equality only holds for 
$(\pp^{i_X-1})^\rho$.
We expect the index to behave quite similarly to $\hat{c}(X,\bM.)$ for a generalized log Calabi--Yau structure.
Thus, it is natural to expect an inequality of the form
\[
\rho(X)(\hat{c}(X,\bM.)-1)\leq d
\]
for generalized klt log Calabi--Yau pairs. This leads to the question: 

\begin{question}
{\em
Is it possible to enrich Conjecture~\ref{conj:gen-complexity} so that it also implies the generalized Mukai conjecture? 
}
\end{question}

\bibliographystyle{habbvr}
\bibliography{bib}

\begin{thebibliography}{10}
\expandafter\ifx\csname url\endcsname\relax
  \def\url#1{\texttt{#1}}\fi
\expandafter\ifx\csname doi\endcsname\relax
  \def\doi#1{\burlalt{doi:#1}{http://dx.doi.org/#1}}\fi
\expandafter\ifx\csname urlprefix\endcsname\relax\def\urlprefix{URL }\fi
\expandafter\ifx\csname href\endcsname\relax
  \def\href#1#2{#2}\fi
\expandafter\ifx\csname burlalt\endcsname\relax
  \def\burlalt#1#2{\href{#2}{#1}}\fi

\bibitem{Amb06}
F.~Ambro.
\newblock The set of toric minimal log discrepancies.
\newblock {\em Cent. Eur. J. Math.}, 4(3):358--370, 2006.
\newblock \doi{10.2478/s11533-006-0013-x}.

\bibitem{Bir19}
C.~Birkar.
\newblock Anti-pluricanonical systems on {F}ano varieties.
\newblock {\em Ann. of Math. (2)}, 190(2):345--463, 2019.
\newblock \doi{10.4007/annals.2019.190.2.1}.

\bibitem{Bir20b}
C.~Birkar.
\newblock Generalised pairs in birational geometry, 2020,
  \burlalt{arXiv:2008.01008}{http://arxiv.org/abs/arXiv:2008.01008}.

\bibitem{BZ16}
C.~Birkar and D.-Q. Zhang.
\newblock Effectivity of {I}itaka fibrations and pluricanonical systems of
  polarized pairs.
\newblock {\em Publ. Math. Inst. Hautes \'{E}tudes Sci.}, 123:283--331, 2016.
\newblock \doi{10.1007/s10240-016-0080-x}.

\bibitem{BMSZ18}
M.~V. Brown, J.~McKernan, R.~Svaldi, and H.~R. Zong.
\newblock A geometric characterization of toric varieties.
\newblock {\em Duke Math. J.}, 167(5):923--968, 2018.
\newblock \doi{10.1215/00127094-2017-0047}.

\bibitem{CLS11}
D.~A. Cox, J.~B. Little, and H.~K. Schenck.
\newblock {\em Toric varieties}, volume 124 of {\em Graduate Studies in
  Mathematics}.
\newblock American Mathematical Society, Providence, RI, 2011.
\newblock \doi{10.1090/gsm/124}.

\bibitem{FFMP22}
F.~Figueroa, S.~Filipazzi, J.~Moraga, and J.~Peng.
\newblock Complements and coregularity of fano varieties, 2022,
  \burlalt{arXiv:2211.09187}{http://arxiv.org/abs/arXiv:2211.09187}.

\bibitem{Fil18b}
S.~Filipazzi.
\newblock Boundedness of log canonical surface generalized polarized pairs.
\newblock {\em Taiwanese J. Math.}, 22(4):813--850, 2018.
\newblock \doi{10.11650/tjm/171204}.

\bibitem{FM20}
S.~Filipazzi and J.~Moraga.
\newblock Strong {$(\delta,n)$}-complements for semi-stable morphisms.
\newblock {\em Doc. Math.}, 25:1953--1996, 2020.

\bibitem{FS20}
S.~Filipazzi and R.~Svaldi.
\newblock On the connectedness principle and dual complexes for generalized
  pairs, 2020,
  \burlalt{arXiv:2010.08018}{http://arxiv.org/abs/arXiv:2010.08018}.

\bibitem{Fuj12}
O.~Fujino.
\newblock Fundamental theorems for semi log canonical pairs.
\newblock {\em Algebr. Geom.}, 1(2):194--228, 2014.
\newblock \doi{10.14231/AG-2014-011}.

\bibitem{FM21}
O.~Fujino and K.~Miyamoto.
\newblock A characterization of projective spaces from the {M}ori theoretic
  viewpoint.
\newblock {\em Osaka J. Math.}, 58(4):827--837, 2021.

\bibitem{Kaw00}
Y.~Kawamata.
\newblock On effective non-vanishing and base-point-freeness.
\newblock In {\em On effective non-vanishing and base-point-freeness}, pages
  173--181. Asian J. Math., 2000.
\newblock \doi{10.4310/AJM.2000.v4.n1.a11}.
\newblock Kodaira's issue.

\bibitem{KO75}
S.~Kobayashi and T.~Ochiai.
\newblock Meromorphic mappings onto compact complex spaces of general type.
\newblock {\em Invent. Math.}, 31(1):7--16, 1975.
\newblock \doi{10.1007/BF01389863}.

\bibitem{KM98}
J.~Koll\'{a}r and S.~Mori.
\newblock {\em Birational geometry of algebraic varieties}, volume 134 of {\em
  Cambridge Tracts in Mathematics}.
\newblock Cambridge University Press, Cambridge, 1998.
\newblock \doi{10.1017/CBO9780511662560}.
\newblock With the collaboration of C. H. Clemens and A. Corti, Translated from
  the 1998 Japanese original.

\bibitem{MS21}
J.~Moraga and R.~Svaldi.
\newblock A geometric characterization of toric singularities, 2021,
  \burlalt{arXiv:2108.01717}{http://arxiv.org/abs/arXiv:2108.01717}.

\bibitem{Nil05}
B.~Nill.
\newblock Gorenstein toric {F}ano varieties.
\newblock {\em Manuscripta Math.}, 116(2):183--210, 2005.
\newblock \doi{10.1007/s00229-004-0532-3}.

\bibitem{Sho92}
V.~V. Shokurov.
\newblock Three-dimensional log perestroikas.
\newblock {\em Izv. Ross. Akad. Nauk Ser. Mat.}, 56(1):105--203, 1992.
\newblock \doi{10.1070/IM1993v040n01ABEH001862}.

\bibitem{Sho00}
V.~V. Shokurov.
\newblock Complements on surfaces.
\newblock In {\em Complements on surfaces}, volume 102, pages 3876--3932. J.
  Math. Sci. (New York), 2000.
\newblock \doi{10.1007/BF02984106}.
\newblock Algebraic geometry, 10.

\end{thebibliography}

\end{document}